\documentclass[hidelinks,onefignum,onetabnum]{siamart220329}

\usepackage{geometry}
\usepackage{verbatim,multicol,setspace}
\usepackage{wrapfig}
\usepackage{graphicx}
\usepackage{url}
\usepackage{algorithm}
\usepackage[noend]{algpseudocode}
\usepackage{xcolor}
\usepackage{amsmath}
\usepackage{amssymb,amsmath,amsbsy}
\usepackage{float}
\usepackage{psfrag,epsfig,graphicx,graphics}
\usepackage{subeqnarray,subfigure,rotating}
\usepackage{fancyhdr}
\usepackage{graphicx}
\usepackage{latexsym}
\usepackage{color}
\usepackage{amsfonts}
\usepackage{url}
\usepackage{amscd} 
\usepackage{tikz} 

\usepackage{enumerate}
\usepackage{verbatim}

\newcommand{\pdt}[1]{\frac{\partial #1}{\partial t}}

\newcommand{\R}{{\mathbb R}}

\newcommand{\cL}{\mathcal L}

\DeclareMathOperator*{\argmax}{\arg\!\max}

\usepackage{color}
\usepackage{url}

\newcommand{\deriv}[2]{\frac{d #1}{d #2}}
\def\BE{\begin{equation}}
\def\EE{\end{equation}}
\def\BS{\begin{split}}
\def\ES{\end{split}}
\def\div{\text{div}}

\newcommand{\tony}[1]{#1}
\newcommand{\panos}[1]{#1}
\newcommand{\panosb}[1]{#1}

\newsiamremark{remark}{Remark}
\newsiamremark{hypothesis}{Hypothesis}
\crefname{hypothesis}{Hypothesis}{Hypotheses}
\newsiamthm{claim}{Claim}
\newsiamremark{example}{Example}

\headers{Using Witten Laplacians to locate index-1 saddle points}{Tony Lelièvre, Panos Parpas}

\title{Using Witten Laplacians to locate index-1 saddle points\thanks{Submitted to the editors \today.
}}

\author{Tony Lelièvre, \thanks{CERMICS, Ecole des Ponts, Inria, Paris, France.
  \email{tony.lelievre@enpc.fr}, \url{cermics.enpc.fr/\~lelievre/home.html}}
\and Panos Parpas\thanks{Department of Computing, Imperial College London, UK.
  \email{p.parpas@imperial.ac.uk},\url{www.doc.ic.ac.uk/\~pp500}}}

\usepackage{amsopn}

\ifpdf
\hypersetup{
  pdftitle={Using Witten Laplacians to locate index-1 saddle points},
  pdfauthor={Tony Lelièvre, Panos Parpas}
}
\fi

\begin{document}

\maketitle

\begin{abstract}
    We introduce a new stochastic algorithm to locate the index-1 saddle points of a function $V:\R^d \to \R$, with $d$ possibly large. This algorithm can be seen as an equivalent of the stochastic gradient descent which is a natural stochastic process to locate local minima. It relies on two ingredients: (i) the concentration properties on index-1 saddle points of the first eigenmodes of the Witten Laplacian (associated with $V$) on $1$-forms and (ii) a probabilistic representation of a partial differential equation involving this differential operator. Numerical examples on simple molecular systems illustrate the efficacy of the proposed approach.
\end{abstract}



\section{Introduction and motivation}

A key goal of molecular simulations is to determine reaction mechanisms. These reaction mechanisms are associated with transition paths between local minima of the potential energy surface, which models the interaction between the atoms. Due to the large number of atoms considered in practice, the potential energy function is typically defined on a high-dimensional space.
In order to simulate the dynamics of the molecular system and thus identify these transition mechanisms, one can either use all-atom molecular dynamics or state-to-state dynamics. The former consists in evolving the positions and velocities of all the atoms as continuous functions of time, following typically thermostated Hamiltonian dynamics such as the Langevin dynamics~\cite{allen-tildesley-87,frenkel-smit-02,tuckerman2010statistical}. 
The latter (a.k.a. kinetic Monte Carlo models~\cite{voter-05}  or Markov State Models~\cite{bowman-pande-noe-14}) directly models the jumps between local minima of the potential energy which are separated by index-1 saddle points.
\panosb{Index-1 saddle points are critical points of the energy function such that the Hessian has exactly one negative eigenvalue. }
Local minima have basins of attractions with intersecting boundaries, and the local minima energy points lying in this intersection are indeed necessarily index-1 saddle points. This latter approach can reach much longer timescales than the former one, since it does not require to simulate the details of the trajectories within the wells of the potential energy function.
In practice, the simulation of  state-to-state dynamics requires to identify the local minima as well as the index-1 saddle points connecting them. 
Then, one can compute the rate constants associated with transitions through the different index-1 saddle points using the so-called Harmonic Transition State Theory and Eyring-Kramers laws~\cite{voter-05,wales2003energy,wales2006energy}. 
The fact that the kinetic Monte Carlo dynamics is a good approximation of the Langevin dynamics can be formalized by studying the so-called exit problem in the small noise (namely small temperature) regime, using for example the large deviations theory~\cite{freidlin-wentzell-84} or the quasi-stationary distribution approach~\cite{di-gesu-lelievre-le-peutrec-nectoux-17}. 
This problem is of general interest to the study of noise-induced transition between metastable states separated by energetic barriers~\cite{hanggi1990reaction}: under the influence of small noise, with high probability, the escape pathway has to go through the neighborhood of an index-1 saddle point.

In view of the computational efficiency of the simulation of the state-to-state dynamics described above, many algorithms have been developed to locate local minima and index-1 saddle points of a potential energy function in high dimension. Finding local minima is a standard optimization task, for which a large number of algorithms exist with a well understood properties both in theory and in practice~\cite{conn2000trust,nocedal1999numerical,kelley1999iterative,fletcher2013practical}.  Finding index-1 saddle points is a much harder task. It is beyond the scope of this paper to provide a comprehensive review of numerical methods to locate index-1 saddle points, see the review papers~\cite{schlegel2011geometry,henkelman2002methods} for some pointers to the large literature on this subject.  Broadly speaking, there are two classes of algorithms: double-ended a.k.a. chain-of-state methods and single-ended a.k.a. surface-walking methods. Examples of algorithms belonging to the former family are the nudged elastic band method~\cite{henkelman2000climbing,henkelman2000improved} or the string method~\cite{weinan2002string}: they can be seen as a numerical version of the mountain pass theorem~\cite{bisgard2015mountain} to identify index-1 saddle points. 
They locate index-1 saddle points along minimizing energy paths connecting two fixed local minima. 
They thus require the input of two minima and the global optimization of a non-differentiable problem (namely the minimization over paths of a max function). 
In particular, the results obtained with standard iterative optimization methods typically highly depend on the initial condition of the optimization procedure, namely the initial path between the two local minima. 
For a mathematical analysis, \panosb{we refer to~\cite{cameron2011string,liu2022convergence,liu2022stability}}.
The algorithms in the latter family (single-ended methods) are sometimes called drag methods~\cite{henkelman2002methods}: they indeed consist in successively updating a one-dimensional drag coordinate towards larger energy, and then minimizing the energy of the system in the codimension 1 hyperplane orthogonal to the drag coordinate. 
Early contributions in that direction are~\cite{crippen1971minimization,cerjan1981finding,wales1989finding}.  
\panos{\textbf{}Examples of such methods include the dimer method 
(also known as the eigenvector-following method)~\cite{munro1999defect,henkelman1999dimer,zhang2012shrinking,zhang2016optimization}}, the Activation Relaxation Technique~\cite{mousseau1998traveling,barkema2001activation,cances2009some} and the Gentlest Ascent Dynamics~\cite{weinan2011gentlest,gao2015iterative,journel2023switched}. 
It is known that the efficiency of these techniques also depends on the initialization procedure. For example, in~\cite{levitt2017convergence}, it is shown that the dimer method converges if initialized in neighborhoods of index-1 saddle points, but may otherwise remain trapped in subsets which do not contain any index-1 saddle points. 

Beyond the references mentioned above, there is a vast literature on index-1 saddle point finding algorithms. 
For example, recent variants of the above algorithms aim at reducing the cost by using surrogate models to estimate the gradient of the potential energy~\cite{denzel2018gaussian,koistinen2017nudged,torres2019low}.
Moreover, here are a few examples of other algorithms which do not conform to the classification presented above, and which we will not discuss further since they are less related to our approach: the rational function optimization method~\cite{banerjee1985search}, the step and slide method~\cite{miron2001step}, the biased gradient squared descent~\cite{duncan2014biased}, simulated annealing methods~\cite{chaudhury1998simulated}, genetic algorithms~\cite{ellabaan2009finding}, and branch-and-bound algorithms~\cite{nerantzis2017enclosure}.
This (incomplete) list illustrates the importance of this challenging problem.

In this work, we introduce an algorithm to locate index-$1$ saddle points, which is based on a stochastic interacting particle system and a selection mechanism. 
The proposed algorithm  can be initialized in any basin of attraction of a local minimum of the potential energy function. The algorithm will then attempt to identify the possible exit pathways. As will become clear below, this interacting particle system approximates the solution to a partial differential equation on vector fields (see Equation~\eqref{eq:FP1} below), which can be seen as a generalization of the Fokker-Planck equation for the overdamped Langevin dynamics associated with the potential energy function of interest. 
More precisely, the method we propose derives from considerations on the long-time behavior of the semi-group associated with the Witten Laplacian on $1$-forms associated with the potential energy function of interest (the semi-group associated with the Witten Laplacian on $0$-forms, namely functions, is nothing but the Fokker-Planck equation mentioned above). 
It indeed can be shown that, in the small temperature regime, the vector fields solutions to the partial differential equation~\eqref{eq:FP1} concentrate on index-$1$ saddle points. 
Using the probabilistic representation of such solutions by an interacting particle system, we are thus able to evolve particles which concentrate on neighborhoods of index-1 saddle points: a simple local search (such as the dimer method) can then be used to precisely identify these index-$1$ saddle points.\panos{The application of the local search algorithm is purely a practical step that enables us to accurately identify saddle points. To provide a clearer understanding of our algorithm, it can be viewed as operating in two distinct phases (which are run iteratively in practice).
During the first phase, the interacting particle system focuses on detecting points that are in proximity to index-1 saddle points. The initial phase may not precisely identify the saddle point because we use a constant positive temperature in the particle system.
To achieve a more accurate solution, the second phase relies on a local search method. By employing this approach, we can improve the precision of the algorithm without having to develop an annealing schedule. Notably, in the vicinity of an index-1 saddle point, the dimer algorithm exhibits a linear convergence rate. Consequently, our algorithm endeavors to combine global information derived from the Witten PDE with the effectiveness of local search methods.}

The algorithm we propose is different from the ones discussed above in two ways. 
First, its mathematical foundation, relying on the concentration on index-$1$ saddle points of the $1$-eigenforms of the Witten Laplacian, seems to be new.
Moreover, it is intrinsically based on an interacting particle system.  
The closest algorithm in the literature is probably the Lyapunov weighted dynamics introduced in~\cite{tailleur2007probing}, even though the selection mechanism used there is different in nature: it consists in looking at how solutions starting from close initial conditions diverge. 

In summary, the novelties of this work are twofold. First, we introduce a (to the best of our knowledge) new paradigm to locate index-1 saddle points, using the long-time behavior of the semi-group associated with the Witten Laplacian on $1$-forms. 
Second, we propose a stochastic representation of this semi-group (see Proposition~\ref{prop:FP1}) which can then be used in a concrete algorithm to numerically compute index-1 saddle points. 
The computational requirements of the proposed method are similar to existing methods, and we report encouraging numerical results for standard benchmark problems. 
Let us however emphasize that the objective of this paper is methodological in nature. 
In particular, we do not perform a thorough numerical comparison of the efficiency of our algorithm with other index-1 saddle point search algorithms, neither do we fully optimize the implementation of the numerical method we propose. 
These aspects are left for future work. 

This is an outline of this work. In Section~\ref{sec:Witten}, we present the main theoretical ingredients of the algorithm we propose. The precise description of the algorithm is provided in Section~\ref{sec:algo}. Section~\ref{sec:numerical} is then devoted to numerical experiments to illustrate the properties of the algorithm on simple 2d test cases and two large-dimensional test cases. Finally, Section~\ref{sec:conc} provides a short summary and a few perspectives.

\section{Index\texorpdfstring{\(-1\)}{-1} saddle points and Witten Laplacian on 
\texorpdfstring{\(1-\)}{1-}forms}\label{sec:Witten}

Let us consider a smooth energy function
\begin{equation}\label{eq:V}
V: \R^d \to \R.
\end{equation}
In all the following, it is assumed that the potential energy function $V$ is a Morse function: the Hessian of $V$ at the critical points of $V$ is assumed to be non-degenerate. 
The number of negative eigenvalues is called the index of the critical point. 
We are interested in the index-1 saddle points of~$V$, and more precisely in stochastic algorithms which would be able to concentrate on such points, in a similar way as the overdamped Langevin dynamics concentrate on local minima. 
In the following, we may sometimes omit to explicitly indicate the indices of the saddle points we consider, since we only deal with index-1 saddle points.

This section is organized as follows. 
First, in Section~\ref{sec:0form}, we recall how a noisy steepest descent algorithm (overdamped Langevin dynamics) can be used to find local minima of~$V$, and we give an interpretation of this fact using the Fokker-Planck partial differential equation associated with the dynamics. 
Then, we explain in Section~\ref{sec:1form} how to generalize the Fokker-Planck equation from functions (0-forms) to vector fields (1-forms) and why this is useful to locate index-1 saddle points. 
Finally, Section~\ref{sec:MC1form} introduces a probabilistic interpretation of this generalized Fokker Planck equation that will be used in order to build stochastic dynamics which concentrate on the index-1 saddle points.

\subsection{Local minima, steepest decent and the Witten Laplacian on 0-forms}\label{sec:0form}

In order to introduce the main idea of the paper, let us start with simple observations concerning local minima (namely critical points with index $0$) and the overdamped Langevin dynamics (gradient descent perturbed by noise).

The overdamped Langevin dynamics, for a given inverse temperature $\beta >0$ are given by,
\BE\label{eq:X}
dX_t=-\nabla V(X_t) \, dt+\sqrt{2\beta^{-1}} \, dB_t
\EE
complemented with an initial condition $X_0 \in \R^d$. Here and in the following, $(B_t)_{t \ge 0}$ denotes a $d$-dimensional  Brownian motion. 
The stochastic process $(X_t)_{t \ge 0}$ is ergodic with respect to Boltzmann-Gibbs measure $Z^{-1} \exp(-\beta V(x)) \, dx$ (where we here assume that $Z=\int_{\R^d} \exp(-\beta V) < \infty$) and therefore, in the small temperature regime $\beta \to \infty$, the stochastic process spends most of its time in the lowest energy local minima of~$V$. 
This remark is used in many optimization algorithms which aim at finding the local minima of~$V$, using gradient descent perturbed by small noise.

It is also possible to interpret the properties of the stochastic process $(X_t)_{t \ge 0}$ in the small temperature regime using the so-called semi-classical asymptotic behaviors of associated partial differential equations. More precisely, let us consider the differential operator, defined on a smooth test function $\rho : \R^d \to \R$:
\BE
\cL^*\rho= \div\left({\rho\nabla V}\right)+\beta^{-1} \Delta\rho.
\EE
The Fokker-Planck equation~\cite{risken1996fokker} gives the evolution of the density $\rho(t,x)$ of the time marginals of $(X_t)_{t \ge 0}$ solution to~\eqref{eq:X}:
\BE\label{eq:FP}
\pdt{\rho}= \cL^* \rho,
\EE
where the initial condition $\rho_0$ is assumed to be the density of $X_0$: $\rho(t,x)$ is the density of the random variable $X_t$, for all time $t\ge 0$. In order to describe the Fokker-Planck dynamics in the small temperature regime ($\beta^{-1} \to 0$), let us recall some information on the spectrum of $\cL^*$ in this regime, using the so-called Witten Laplacian on 0-forms, i.e. real-valued functions.

The Witten Laplacian on $0$-forms associated with $V$ and the small parameter $h=2 \beta^{-1}$ is defined by, for a smooth test function $u:\R^d \to \R$:
\begin{equation}\label{eq:Witten0}
    \Delta^{(0)}_{V,h} u = -h^2 \Delta u + (|\nabla V|^2 - h \Delta V) u.
\end{equation}
There is a simple link between $\cL^*$ and $\Delta^{(0)}_{V,h}$: for any smooth test function $u$,
\begin{equation}\label{eq:unitary}
\exp(V/h) \cL^* (u \exp(-V/h)) = -\frac{1}{2h} \Delta^{(0)}_{V,h} u.
\end{equation}
In particular, the spectral properties of $\cL^*$ can be deduced from the spectral properties of $\Delta^{(0)}_{V,h}$ in $L^2(\R^d)$:
$$- \cL^* \rho = \lambda \rho \iff \Delta^{(0)}_{V,h} u = 2 h \lambda u,$$
with $u=\rho \exp(V/h)$. \tony{Notice that Equation~\eqref{eq:Witten0} is nothing but a Schr\"odinger operator with the $h$-dependent potential $|\nabla V|^2 - h \Delta V$.}

It is easy to check\footnote{This is a consequence of the fact that $\Delta^{(0)}_{V,h}= d^{(0),*}_{V,h} d^{(0)}_{V,h}$ where $d^{(0)}_{V,h}=e^{-V/h} (h \nabla) e^{V/h} $ and $d^{(0),*}_{V,h}=-e^{V/h} (h {\rm div}) e^{-V/h}$ is the $L^2$-adjoint of $d^{(0)}_{V,h}$.} that operator $\Delta^{(0)}_{V,h}$ is non negative and symmetric for the $L^2(\R^d)$ scalar product. Under some generic assumptions on $V$ (one of them being that $V$ is a Morse function), one can prove the following properties on $\Delta^{(0)}_{V,h}$ (we refer to~\cite{simon1983semiclassical,helffer1985puits} or~\cite[Theorem 3.9]{le2013precise} for precise statements on the eigenvalues and to~\cite{helffer2006semi,helffer1985puits2,dimassi1999spectral} for precise statements on the eigenfunctions\footnote{\label{fn}As mentioned in~\cite{simon1983semiclassical}, these results are well known
in the folk wisdom of theoretical physics, and it is beyond the scope of this work to state mathematically precise asymptotic results. This would indeed require to introduce many notation and notions, such as the Agmon distance or Morse coordinates. We prefer to stick to a gentle presentation which is sufficient for our purpose. The essential results underpinning our algorithms will be supported by numerical illustrations when we will consider $1$-forms in the next section, see Example~\ref{example:double_well}.}):
\begin{itemize}
    \item On the eigenvalues: For $h$ sufficiently small, the operator $\Delta^{(0)}_{V,h}$ admits exactly $m^{(0)}$ eigenvalues smaller than $ch$ (for some $c>0$), where $m^{(0)}$ is the number of local minima of~$V$. Moreover, these eigenvalues are exponentially small when $h \to 0$, and the rest of the spectrum is bounded from below by a constant independent of $h$.
    \item On the eigenfunctions: The eigenfunctions associated with these $m^{(0)}$ small eigenvalues are essentially supported in neighborhoods of the local minima of $V$. More precisely, one can prove so-called Agmon-type decay estimate which shows that the eigenvectors decay exponentially fast away from the local minima, in the Agmon distance. 
    In particular, the first eigenvalue is $\lambda_1=0$ associated with the ($L^2$-normalized) eigenfunction $u_1=\exp(-V/h)/\sqrt{\int_{\R^d} \exp(-2V/h)}$.
\end{itemize}
The qualitative properties above, give a clear picture of the evolution in time of $\rho$ i.e. the solution to~\eqref{eq:FP}. Indeed, in the small temperature regime, by projecting the evolution on the dominant eigenmodes $(u_k)_{1 \le k \le m^{(0)}}$ and using~\eqref{eq:unitary}, one has (assuming without loss of generality that the family of functions $(u_k)_{1 \le k \le m^{(0)}}$ is orthonormal in $L^2$): for all $t \ge 0$,
\begin{equation}\label{eq:expansion}
\rho(t) \simeq \sum_{k=1}^{m^{(0)}} \exp\left(-\frac{\lambda_k}{2h} t\right) \, \left\langle \rho_0 \exp\left(  V / h \right) ,u_k \right\rangle \, u_k \exp\left(-  V / h\right)
\end{equation}
the remaining terms decaying much faster in time than these first $m^{(0)}$ terms since they are associated with eigenvalues of $\Delta^{(0)}_{V,h}$ (or equivalently of $-\cL^*$) \tony{which are bounded from below by a constant times $h$, and thus much larger than $\lambda_1, \ldots, \lambda_{m^{(0)}}$}. Here, $$\langle \rho_0 \exp(V/h) ,u_k \rangle=\int_{\R^d} \rho_0 \exp(V/h) u_k$$
is the coefficient associated with the $L^2(\R^d)$ orthogonal projection of $\rho_0 \exp(V/h)$ on the $k$-th eigenmode $u_k$. Notice that the first term in the sum (for $k=1$) is simply $ \exp(-2V/h)/\int_{\R^d} \exp(-2V/h) = Z^{-1} \exp(-\beta V)$, which is consistent with the fact that the longtime limit of $\rho$ is the Boltzmann-Gibbs measure associated with $V$.  Moreover, one infers from~\eqref{eq:expansion}  that before convergence, the density~$\rho$ is essentially a linear combination of the functions $u_k \exp(-V/h)$ which are concentrated around the local minima of $V$. This can be seen as another illustration of the fact that the overdamped Langevin dynamics~\eqref{eq:X} spends most of its time around the local minima of~$V$.

We will now show in the next section how a similar relationship between a stochastic differential equation and partial differential equation can be used in order to visit the index-1 saddle points of~$V$.

\subsection{Index-1 saddle points and the Witten Laplacian on 1-forms}\label{sec:1form}

Similarly to the Witten Laplacian $\Delta^{(0)}_{V,h} $ on $0$-forms (see~\eqref{eq:Witten0}), one can introduce the Witten Laplacian on $1$-forms, namely on vector fields\footnote{\tony{We here identify 1-forms with vector fields using the fact that the underlying scalar product is the Euclidean one: the components of $v$ are thus the components of the $1$-form on the basis $dx_1$, ... $dx_d$.}} $v:\R^d \to \R^d$
\begin{equation}\label{eq:Witten1}
    \Delta^{(1)}_{V,h} v = -h^2 \Delta v + (|\nabla V|^2 - h \Delta V) v + 2h \nabla^2 V v,
\end{equation}
where $\nabla^2 V$ denotes the Hessian of $V$, and the last term is thus a matrix-vector product. 
Similarly to $\Delta^{(0)}_{V,h}$, the operator $\Delta^{(1)}_{V,h}$  is non negative and symmetric 
\footnote{Again, this is a consequence of the fact that $\Delta^{(1)}_{V,h}= d^{(1),*}_{V,h} d^{(1)}_{V,h} + d^{(1)}_{V,h} d^{(1),*}_{V,h}$ where $d^{(1)}_{V,h}=e^{-V/h} h d^{(1)} e^{V/h} $ with $d^{(1)}$ the differential operator (exterior derivative) on $1$-forms, and $d^{(1),*}_{V,h}$ is the $L^2$-adjoint of $d^{(1)}_{V,h}$.} 
in $(L^2(\R^d))^d$. 
Under some generic assumptions on $V$, one can prove the following properties on $\Delta^{(1)}_{V,h}$ (we again refer to~\cite{helffer1985puits,le2013precise,helffer2006semi,helffer1985puits2,dimassi1999spectral} for precise statements and generalizations to Witten Laplacian on $k$-forms and index-$k$ saddle points, see 
also the footnote \textsuperscript{\ref{fn}}).
\begin{itemize}
    \item On the eigenvalues: For $h$ sufficiently small, the operator $\Delta^{(1)}_{V,h}$ admits exactly $m^{(1)}$ eigenvalues smaller than $ch$ (for some $c>0$), where $m^{(1)}$ is the number of index-1 saddle points of $V$. Moreover, these eigenvalues are exponentially small when $h \to 0$, and the rest of the spectrum is bounded from below by a constant independent of $h$.
    \item On the eigenforms: The eigenforms associated with these $m^{(1)}$ small eigenvalues are concentrated in neighborhoods of the index-1 saddle points of $V$: \tony{this means that the norms of these eigenforms (or equivalently of the associated vector fields) are essentially supported in those neighborhoods}. More precisely, one can prove so-called Agmon-type decay estimate which shows that the norms of the eigenforms decay exponentially fast away from the index-$1$ saddle points, in the Agmon distance. \tony{Moreover, in the vicinity of a saddle point where it concentrates, the eigenform  is essentially aligned in the direction of the eigenvector associated with the negative eigenvalue of the Hessian of $V$ at this saddle point}. 
\end{itemize}
In view of these theoretical properties and mimicking the considerations of Section~\ref{sec:0form} on algorithms to find local minima of $V$, it is now tempting to try and find a probabilistic interpretation of the partial differential equation:
\BE \label{eq: witten pde}
\frac{\partial v}{\partial t} + \Delta^{(1)}_{V,h} v  = 0
\EE
since this should provide an algorithm to find index-1 saddle points of $V$.

To do so, let us first use again the transformation~\eqref{eq:unitary} (with, as before $h=2\beta^{-1}$), and let us thus introduce the operator $\tilde \cL^*$ defined by: for any test function $v:\R^d \to \R^d$
$$[\tilde \cL^* v]_i = \cL^* v_i - [\nabla^2 V v]_i$$
where here and in the following, the subscript $_i$  (with $1 \le i \le d$) denotes the $i$-th component of a vector in~$\R^d$. One can check that
\begin{equation}\label{eq:unitary1}
\exp(V/h) \tilde \cL^* (v \exp(-V/h)) = -\frac{1}{2h} \Delta^{(1)}_{V,h} v.
\end{equation}
As in the previous section, the spectral properties of $\tilde \cL^*$ can be deduced from the spectral properties of $\Delta^{(1)}_{V,h}$. In particular, the solution to the partial differential equation (which is the equivalent on $1$-forms to the Fokker-Planck equation~\eqref{eq:FP} on $0$-forms)
\BE\label{eq:FP1}
\pdt{\phi}= \tilde \cL^* \phi
\EE
supplemented with an initial condition $\phi_0$, should be such that it concentrates on index-1 saddle points of $V$ as time increases, since the dominating eigenforms concentrate on those saddle points (see Example~\ref{example:double_well} below for an illustration on a 2d toy example). Here, $\phi$ is a function of $t>0$ and $x \in \R^d$, and takes values in $\R^d$.  In the following, we will refer to the partial differential equation~\eqref{eq:FP1} as the Witten partial differential equation.

Of course, solving directly the Witten partial differential equation~\eqref{eq:FP1} is only possible for low-dimensional toy potentials (such as the one used in Example~\ref{example:double_well} below). 
However, just like the overdamped Langevin dynamics in 
\eqref{eq:X} offer a way to approximate the solution of the Fokker-Planck equation in high-dimensions we can legitimately ask if there is a similar stochastic representation of the solution to the Witten partial differential equation. The stochastic dynamics will of course only yield a noisy approximation  but will also offer a way to compute saddle points in high dimensions. We provide such a stochastic interpretation of the Witten partial differential equation~\eqref{eq:FP1} in the next section.




\begin{example}\label{example:double_well}
In order to illustrate the theoretical properties of the differential operators we presented above, let us close this section with a numerical comparison between the longtime behaviors of the Fokker-Planck partial differential equation~\eqref{eq:FP} and the Witten partial differential equation~\eqref{eq:FP1} on a 2d toy problem, with the double well potential:
\begin{equation}\label{eq:double_well}
V(x_1,x_2)=E(Cx_1^4-x_1^2)+\mu x_2^2,
\end{equation}
where $E=2\times10^{-4}, \ C=0.045, \text{ and } \mu = 0.001$. The potential energy function $V$ admits two local minima at 
$[\pm1/2C,0]$ and a single index-1 saddle point at $[0,0]$. We solve both equations with an explicit finite difference scheme.
 We use the following initial conditions: 
$\rho_0(x_1,x_2)=\exp(-((x_1-0.5)^2+(x_2-0.5)^2)/\sigma_0)$ with $\sigma_0=0.001$, and $\phi_0(x_1,x_2)=[\rho_0(x_1,x_2),\rho_0(x_1,x_2)]$, and the inverse temperature parameter $\beta=10^3$.
As expected the solution of the Fokker-Planck equation concentrates equally on the two local minima (see Figure \ref{fig:pde fokker}).
On the other hand, the first component of the solution of the Witten partial differential equation, denoted by $\phi_1(x_1,x_2)$ (left of Figure \ref{fig:pde vec}) concentrates on the only saddle point of the function at $[0,0]$, while the second component, denoted by $\phi_2(x_1,x_2)$ (right of Figure \ref{fig:pde vec}) eventually vanishes. Notice the difference in scale between the two plots for  $\phi_1(x_1,x_2)$ and $\phi_2(x_1,x_2)$.
We thus observe that, as expected, in the longtime limit, the solution of~\eqref{eq:FP1}  concentrates on the index-1 saddle point of the potential, and aligns along the eigenvector associated with the negative eigenvalue of the Hessian of $V$ at the saddle point. This illustrates the fact that the solution of the Witten partial differential equation can indeed be used to locate the index-1 saddle points of the potential $V$. \tony{Let us insist on the fact that this concentration phenomenon occurs whatever the initial condition: this is at variance with other saddle points localization techniques whose convergence properties depend on the initialization, as discussed in the introduction.}
\end{example}

 \begin{figure}\label{fig: fokker witten}
\centering  
\subfigure[Fokker-Planck partial differential equation for a double-well potential.]{\label{fig:pde fokker}
\includegraphics[width=0.4\textwidth]{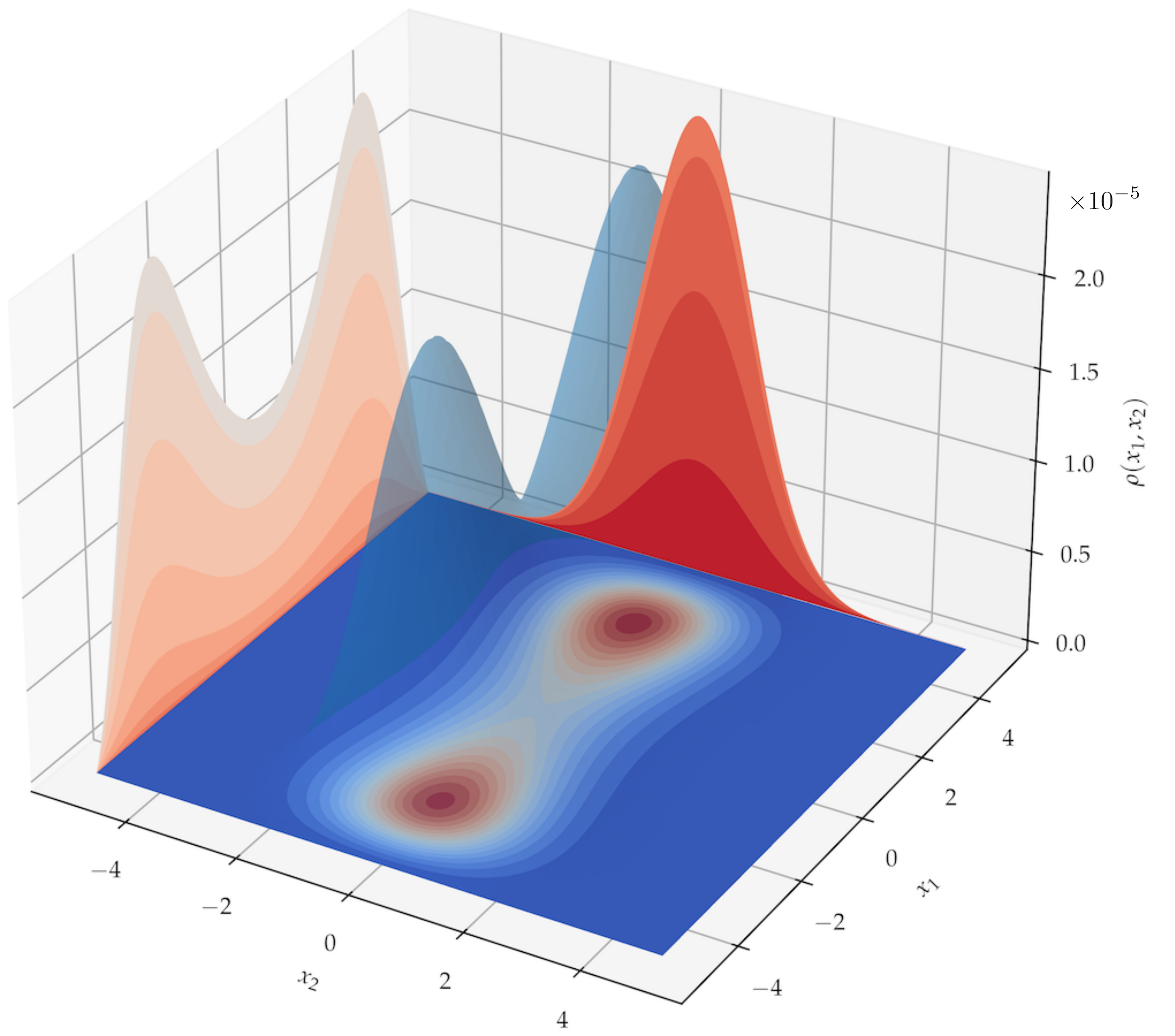}
} 
\subfigure[Witten partial differential equation \panosb{for} a double-well potential.]{\label{fig:pde vec}\includegraphics[width=0.87\textwidth]{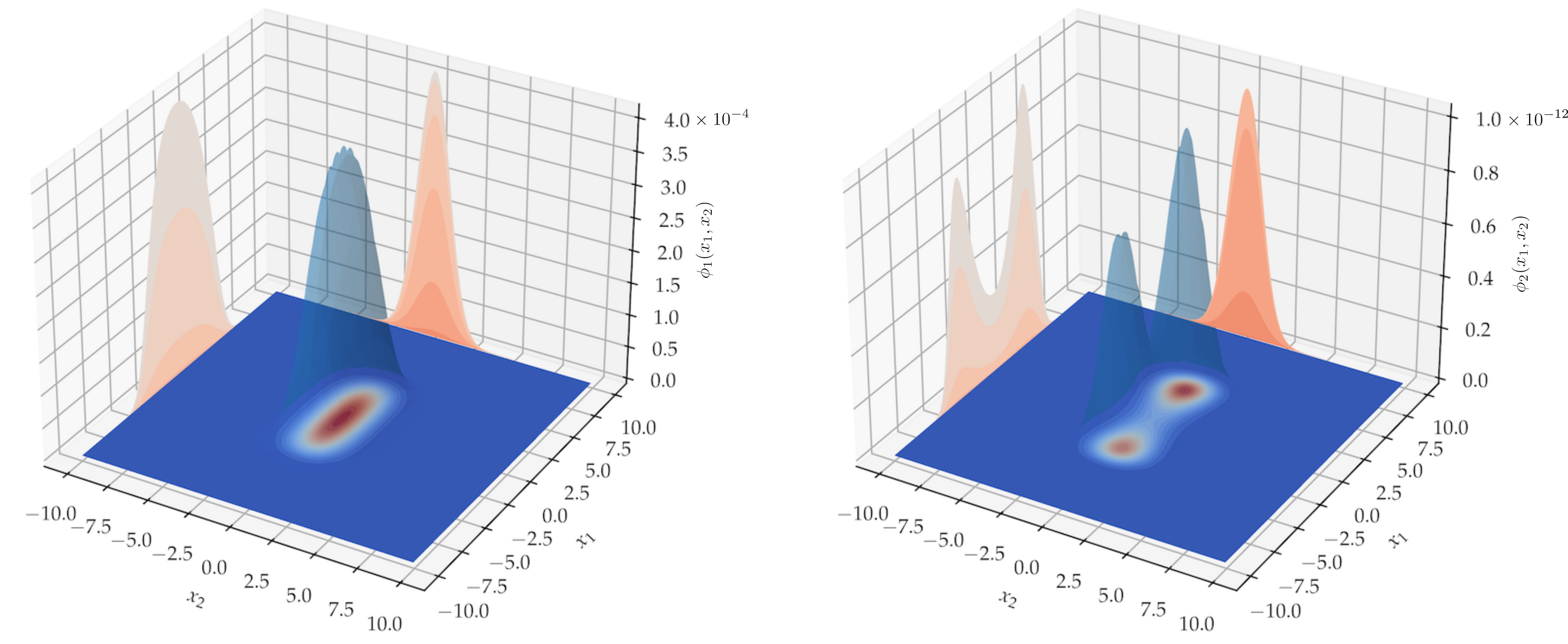}}
\caption{(a) The solution of the Fokker-Planck equation~\eqref{eq:FP}  for the double-well potential~\eqref{eq:double_well}. 
(b) The solution of the Witten partial differential equation~\eqref{eq:FP1} for the same potential ($\phi_1$ is represented on the left plot, and $\phi_2$ on the right one). All these solutions are plotted at a large time $t$, close to stationary state.
We provide contour plots and 3d representations of the three functions, together with 2d projections along the $x_1$ and $x_2$ axes: for example, the projections in red on (a) represent $\int\rho(x_1,x_2)dx_2$ and $\int\rho(x_1,x_2)dx_1$.
An animation associated with this Figure is available in the supplementary material.
} 
\end{figure}

\subsection{A probabilistic interpretation of the partial differential equation
\texorpdfstring{(\ref{eq:FP1})}{(\ref{eq:FP1})}
}\label{sec:MC1form}

In order to give a probabilistic interpretation to~\eqref{eq:FP1}, we take inspiration from the works in~\cite{elworthy1994formulae,elworthy2008l2}  which introduce the stochastic flow associated with~\eqref{eq:X} in order to draw links between differential operators on differential forms and stochastic processes. Let us introduce the overdamped Langevin dynamics~\eqref{eq:X} augmented with second order information:
\BE\label{eq:sde full}
\left\{
\begin{aligned}
 dX_t&=-\nabla V(X_t)\, dt+\sqrt{2\beta^{-1}} \, dB_t, \\
 dY_t&=-\nabla^2 V(X_t)Y_t \, dt, 
\end{aligned}
\right.
\EE
complemented with an initial condition $(X_0,Y_0)$. Here, both $(X_t)_{t \ge 0}$ and $(Y_t)_{t \ge 0}$ are stochastic processes with values in $\R^d$. The Fokker-Planck equation associated with~\eqref{eq:sde full} is given by,
\begin{equation}\label{eq:kappa}
\frac{\partial \kappa}{\partial t} = \div_x(\nabla V(x) \kappa + \beta^{-1} \nabla_x \kappa) + \div_y (\nabla^2 V (x) y \kappa)
\end{equation}
where for any $t \ge 0$, $\kappa(t,x,y)dx dy$ is the law of $(X_t,Y_t)$. Let us now introduce the function $\phi:\R_+ \times \R^d \to \R^d$
\begin{equation}\label{eq:phi}
\phi(t,x) = \int_{\R^d} y \kappa(t,x,y) \, dy.
\end{equation}
The main result of this section is the following proposition.
\begin{proposition}\label{prop:FP1}
The function $\phi$ defined by~\eqref{eq:phi} satisfies the Witten partial differential equation~\eqref{eq:FP1}.
\end{proposition}
\begin{proof}
This result follows from a straightforward computation. Indeed, for $i \in \{1, \ldots,d\}$,
\begin{align*}
    \partial_t \phi_i &=
    \int_{\R^d} y_i \partial_t \kappa(t,x,y) \, dy\\
     &=\int_{\R^d} y_i \left( \div_x(\nabla V(x) \kappa + \beta^{-1} \nabla_x \kappa) + \div_y (\nabla^2 V (x) y \kappa) \right) \, dy\\
    &=  \div_x\left(\nabla V(x) \int_{\R^d} y_i \kappa \, dy + \beta^{-1} \nabla_x \int_{\R^d} y_i\kappa \, dy \right) - \int_{\R^d} \nabla_y y_i \cdot (\nabla^2 V (x) y) \kappa \, dy \\
    &=\div_x (\nabla V(x) \phi_i) + \beta^{-1} \Delta_x \phi_i - [\nabla^2 V(x) \phi ]_i 
\end{align*}
which concludes the proof.
\end{proof}

\begin{remark}
\tony{Notice that we have assumed for simplicity that $(X_t,Y_t)$ has a density~$\kappa$ to write the argument above, without making precise assumptions under which this would be true. If it is not the case, one can use a similar reasoning using a weak formulation of the Fokker Planck equation~\eqref{eq:kappa} and of the Witten partial differential equation~\eqref{eq:FP1}. Notice that since the latter is elliptic, it has regularizing effects, which make $\phi$ smooth for any positive time, at least for a smooth potential $V$.}
\end{remark}

Our objective is now to propose Monte Carlo approximations of $\phi(t,x)$, using the probabilistic interpretation~\eqref{eq:phi}. The basic idea is to use many independent realizations $(X^n_t,Y^n_t)_{1 \le n \le N}$ of the dynamics~\eqref{eq:sde full}, and to approximate $\phi(t,x)$ by an empirical distribution:
$$\phi(t,x) \, dx \simeq \frac{1}{N} \sum_{n=1}^N Y^n_t \delta_{X^n_t}(dx),$$
relying on the approximation of the law of $(X_t,Y_t)$ by the empirical distribution \linebreak $\frac{1}{N} \sum_{n=1}^N  \delta_{(X^n_t,Y^n_t)}(dx \, dy)$.
This can be seen as a weighted average of the Dirac masses $\delta_{X^n_t}$, and one therefore expects a large variance if, at a given time $t$, the  vectors $(Y^n_t)_{1 \le n \le N}$ vary a lot. We therefore propose to rewrite this empirical distribution as follows (where $\|.\|$ denotes here and in the following the Euclidean norm):
$$\frac{1}{N} \sum_{n=1}^N Y^n_t \delta_{X^n_t} = \frac{1}{N} \sum_{n=1}^N \|Y^n_t\| \frac{Y^n_t}{\|Y^n_t\|} \delta_{X^n_t}, $$
and to perform resampling using the weights,
$$w^n_t=\|Y^n_t\|,$$
at regular time intervals in order to have a good sampling of the regions associated with large values of $\|Y_t\|$. Indeed, we know from~\eqref{eq:phi} and the discussion at the beginning of this section that these regions (where the support of $\phi$ concentrates) are neighborhoods of index-1 saddle points. This resampling strategy is very much in the spirit of sequential importance sampling~\cite{doucet2001sequential,del2004feynman}.

This leads to the basic algorithm we are using in this work: the Stochastic Saddle Point Dynamics, see Algorithm~\ref{alg:cap}. Notice that here and in the following, we use the notation
$$\mathrm{w}= (w^1, \ldots, w^N ),$$
to denote the vector of the $N$ weights.
The algorithm proceeds as follows: at each iteration~$k$,
\begin{itemize}
    \item If the effective sample size of the weights $\mathrm{w}_k=(w^1_k, \ldots, w^N_k)$ is larger than some threshold $\rho_{ess} N$ ( where $\rho_{ess} \in (0,1)$) then a resampling step is performed;
    \item Perform one step of an Euler-Maruyama discretization of~\eqref{eq:sde full} to update $(X^n_k,Y^n_k,w^n_k)$, 
    ${n=1,\ldots,N}$, using i.i.d. centered reduced Gaussian random variables $(G^n_k)_{ 1 \le n \le N}$. 
\end{itemize}

The effective sample size (ESS) of the weights $\mathrm{w}= (w^1, \ldots, w^n )$ is defined as follows:
$$\text{ESS}(\mathrm w) = \frac{\left(\sum_{n=1}^N w^n\right)^2}{\sum_{n=1}^N (w^n)^2}.$$
This real number lies in the interval $[1,N]$. 
If all the weights are equal, then $\text{ESS}(\mathrm w)=N$, and the sampling scheme collapses to a simple Monte Carlo scheme. If all the weights but one are 0 then $\text{ESS}(\mathrm w)=1$.
The larger the discrepancy between the weights, the smaller the effective sample size.

The resampling routine returns a set of size $N$, by drawing randomly with replacement from $\{1,\ldots,N\}$. We denote the set computed by the resampling routine by $J$. We use the notation $X^{n}\leftarrow X^{J[n]}$ to denote that the $n^{\text{th}}$ particle is replaced by the particle indicated by the $n^{\text{th}}$ entry of $J$. Notice that any resampling scheme could potentially be used. In our numerical experiments, we used stratified resampling \cite{chopin2020introduction}. 
In practice, we found that the proposed scheme is robust to the choice of the resampling scheme and the ESS threshold parameter over a large range of values.

We use a simple explicit Euler-Maruyama scheme for the discretization of the system of stochastic differential equations~\eqref{eq:sde full}.
Other discretization schemes could potentially improve the performance of the algorithm, especially if the problem is not well-conditioned, but the time discretization is not a focus of this paper. In terms of computational requirements we note that the algorithm does not need to compute the \panosb{Hessian} of $V$ since we only need to compute products of the Hessian matrix with vectors. We use an algorithmic differentiation library (see Section \ref{sec:numerical} for details) to compute $\nabla V(x)$ and 
    $\nabla( y^\top V(x))=\nabla^2 V(x)y$ at the same time with a computational cost of $O(d)$ for each particle.  
    We also note that the only part of the algorithm that cannot be parallelized is the resampling step (which is negligible in terms of computational cost).
    Thus the algorithm lends itself to an embarrassingly parallel implementation (which we fully take advantage~of).

\begin{algorithm}
\caption{Stochastic Saddle Point Dynamics (SSPD)}\label{alg:cap}
\begin{algorithmic}[1]
\Require $K>1$, $\delta >0$, $N\geq 1$, $\mathrm X_0$, $\mathrm Y_0$, $\rho_{ess}\in(0,1)$ 
\For{$k=0,\ldots,K$}
\State $w^n_{k}=\|Y^n_{k}\|$ \hspace{5.8cm} $n=1,\ldots,N$ 
\If{$\text{ESS}(\mathrm w_{k})\leq \rho_{ess} N$}
 \hspace{3.5cm}\text{\Comment{Run resampling algorithm}} 
\begin{align*}
&J\leftarrow {\tt resample}(\mathrm w_k)& \ & \  \ \ \  \text{\Comment{Resample particles}} \\
 & X^n_k\leftarrow X_k^{J[n]}& \ & \  \ \ \ n=1,\ldots,N \\
& Y^n_k \gets Y^{J[n]}_k/w^{J[n]}_k & \ & \  \ \ \ n=1,\ldots,N 
 \end{align*}
\EndIf 
\For{$n=1,\ldots,N$} 
\State 
$X^n_{k+1}= X^n_{k}-\delta \, \nabla V (X^n_{k}) +\sqrt{2\beta^{-1} \delta} \, G^n_k$ \hspace{0.8cm}\text{\Comment{Update States}}
\State $Y^n_{k+1}= Y^n_{k}-\delta \, \nabla^2 V (X^n_{k}) Y^n_k $ 
\EndFor
\EndFor
\end{algorithmic}
\end{algorithm}
\begin{figure}[ht]
\centering     
\subfigure[]{\label{fig:gd example 1}
\includegraphics[width=0.45\textwidth]{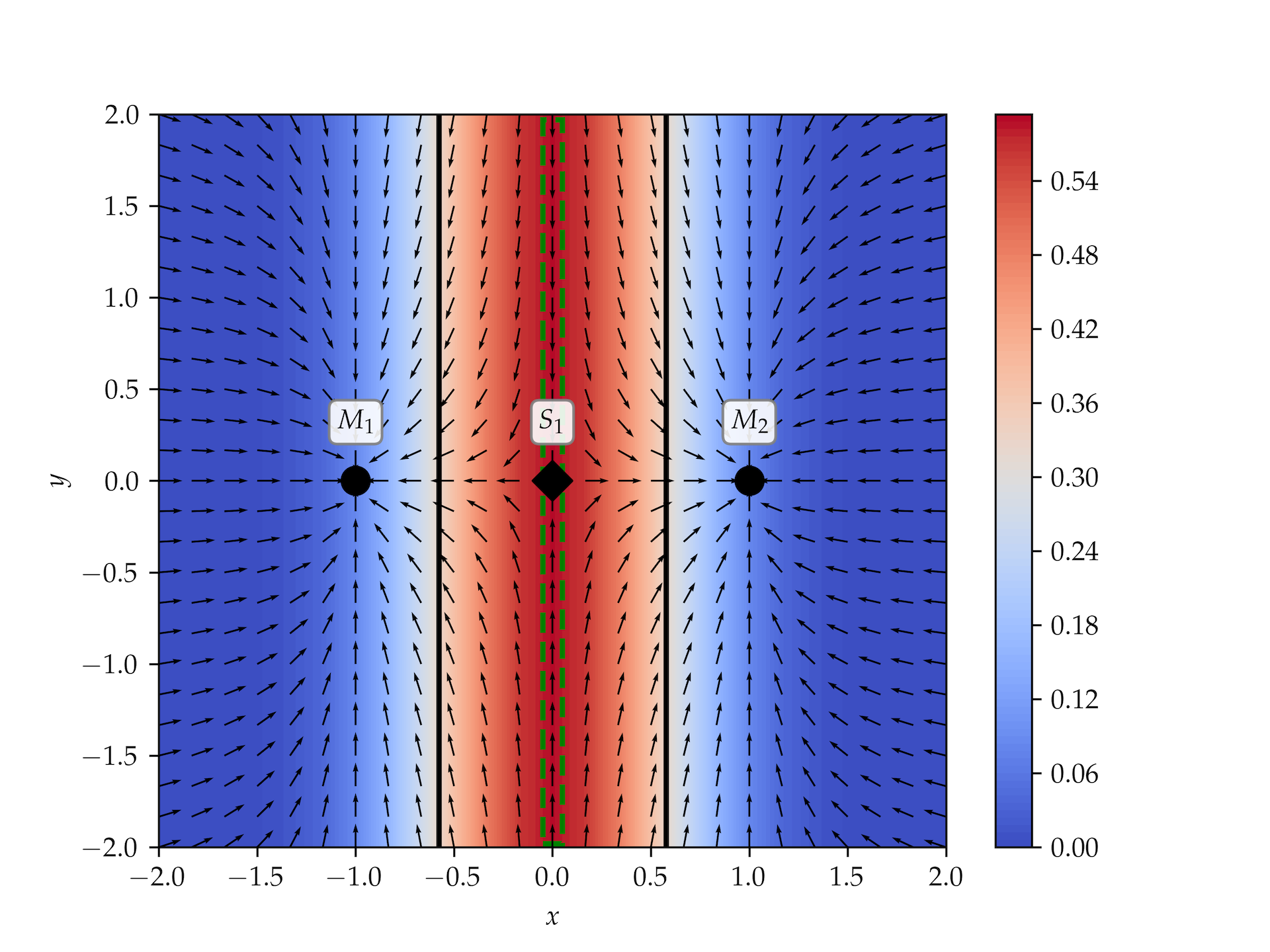}}
\subfigure[]{\label{fig:idimer example 1}
\includegraphics[width=0.45\textwidth]{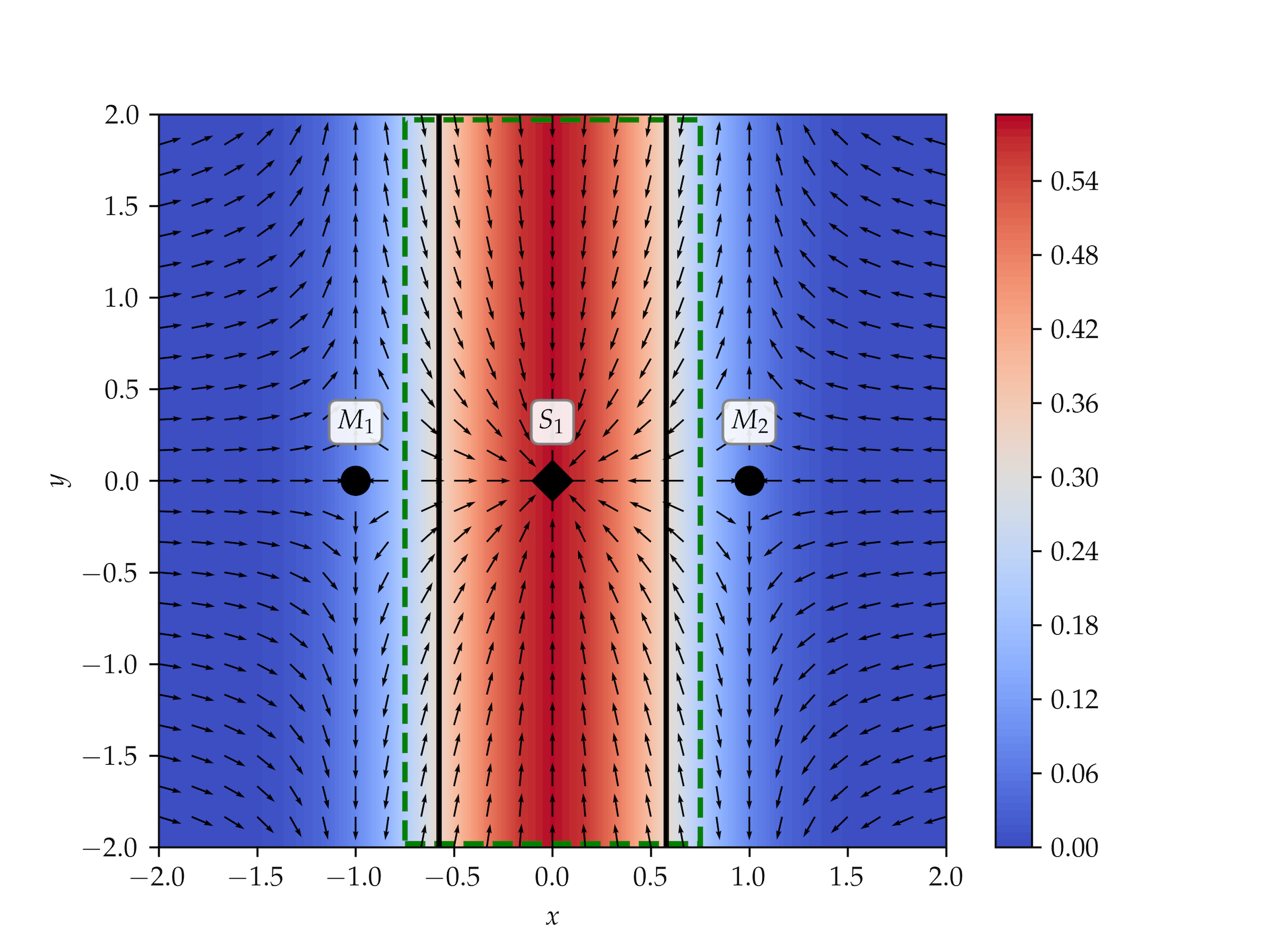}}
\caption{See Example~\ref{example: toy2d}. The heatmap represents the resampling weights from Algorithm 1. The arrows on Figure (a) shows the vector field $-\nabla V$, and on Figure (b) the vector field associated with the idealized dimer dynamics~\eqref{eq: ideal dimer}. 
}
\end{figure}
\begin{example}\label{example: toy2d}
In order to illustrate the role of the weights $w^n_t=\|Y^n_t\|$, let us consider a simple example.
In Figure \ref{fig:gd example 1} we plot the gradient vector field of the double-well potential, $$V(x,y)=(1-x^2)^2+2y^2,$$ that has two global minima at $[\pm 1,0]$ marked by $M_1$ and $M_2$ respectively and a  saddle point at $[0,0]$ marked by $S_1$. 
The Hessian of this function is positive definite everywhere, apart from when $-1/\sqrt3\leq x \leq 1/\sqrt3$ (the boundary of this region is indicated by the two vertical black lines). Starting from $Y_0=[1,1]^\top$ and an arbitrary initial condition $X_0$, and noting that the Hessian is a diagonal matrix, the weight of particle $i$ at time $\delta$ is,
\BE
w^i_\delta (X_0)=\|Y^i_\delta\|= \|\exp(-\nabla^2 V(X_0)\delta)Y^i_0\|+O(\delta^2).\label{eq:approx weights}
\EE
The heatmap in Figure \ref{fig:gd example 1} represents the weights as given by   \eqref{eq:approx weights} which is valid for a small enough step-size $\delta$. We thus see, as expected, that when the particle visits a region with a positive eigenvalue then its weight will be reduced, whereas regions with negative eigenvalues increase the weight of the particle. After the resampling step, particles in the red region
are more likely to survive than particles in the blue region.
In Figure \ref{fig:gd example 1} we also show the vector fields of the negative gradient, i.e. the deterministic part of the dynamics on $X_t$ in~\eqref{eq:sde full}.
This figure captures the intuition behind Algorithm \ref{alg:cap} and the dynamics in \eqref{eq:sde full}. The first equation in \eqref{eq:sde full} drives the particles towards local minima, while the resampling scheme favors the particles that explore regions of the space likely to have saddle points. 
\end{example}
\section{An algorithm combining global search with the Witten Laplacian on \texorpdfstring{\(1-\)}{1-}forms and local search with the dimer method}\label{sec:algo}

The Example~\ref{example: toy2d} is encouraging because it shows that the weights of the resampling scheme have the correct behavior. 
Actually, we will check below that the algorithm indeed works as expected for this simple example (see Section \ref{sec:double well}).

However, for larger dimensional problems, we have observed that it is difficult to tune the parameter~$\beta$ to make
Algorithm~\ref{alg:cap} effective. Indeed, on the one hand, from a theoretical viewpoint (see Section~\ref{sec:1form}), one requires $\beta$ to be as large as possible in order to concentrate the solution to the Witten partial differential equation~\eqref{eq:FP1} on the saddle points. But on the other hand, from a numerical viewpoint, for large $\beta$, the dynamics on $X_t$ in \eqref{eq:sde full} very quickly brings the particles to local minima of $V$, and the resampling strategy has difficulties to keep the index-1 saddle points populated. Therefore, we propose to use Algorithm~\ref{alg:cap} with a not too large $\beta$, in order for the selection mechanism to keep particles around index-1 saddle points, and in a post-processing step, to periodically perform a local saddle point search method~\cite{henkelman1999dimer,zhang2012shrinking,zhang2016optimization} on the particles to precisely locate the saddle points\footnote{It would be worth exploring another natural option to address this difficulty, namely using an annealing schedule, with a time-increasing $\beta$. This is out of the scope of this work.}. 
In practice, we use this local saddle point search only on particles which are in a region where exactly one eigenvalue of the Hessian of $V$ is negative, and the other ones are positive, since we know that local saddle point search algorithms may not converge when initialized far away from saddle points (this will be discussed below). For later use, we denote this so-called index-1 region by
\BE\label{eq:define set S}
S=\{x\in\R^d~|~ \lambda_1(x)<0<\lambda_2(x)\}
\EE
where $\lambda_1(x)\leq \lambda_2(x)\leq \ldots\leq \lambda_d(x)$ are the ordered eigenvalues of the Hessian of~$V$ at point~$x$, and where ties are resolved arbitrarily.

We use the dimer method~\cite{henkelman1999dimer,zhang2012shrinking,zhang2016optimization} as the local index-1 saddle point search (of course, our algorithm could be combined with any other local saddle point search algorithm, such as a Newton method for example).
The details of the dimer algorithm are presented in Section~\ref{sec:algo1_dimer}, while a variant of it, more suitable to situations with many saddle points, is introduced in Section~\ref{sec:algo1_pdimer}. We postpone numerical illustrations of these algorithms to Section~\ref{sec:numerical}.

\subsection{Combining Algorithm~\ref{alg:cap} with the dimer method}\label{sec:algo1_dimer}

Let us briefly introduce the dynamics of the dimer method~\cite{henkelman1999dimer},
basing our discussion on the idealized dimer (i.e. the original method but in continuous time and in the limit of a dimer of zero length~\cite{levitt2017convergence}). The dimer dynamics on $t \mapsto u(t) \in \mathbb R^{d}$ are given by:
\BE\label{eq: ideal dimer}
\deriv{u}{t}=-(I_d-2v_1(u) v_1(u)^\top)\nabla V(u),
\EE
where, for any $u \in \mathbb R^d$, $v_1(u)$ is the normalized eigenvector associated with the lowest eigenvalue of~$\nabla^2 V(u)$. Here and in the following, $I_d$ denotes the $\mathbb R^{d \times d}$ identity matrix. It has been shown in~\cite[Theorem 2.2]{levitt2017convergence} that the dynamics~\eqref{eq: ideal dimer} converge to an index-1 saddle point if the initial condition is close enough to a saddle point. 
Returning to Example~\ref{example: toy2d}, we see in Figure~\ref{fig:idimer example 1} that the set of  initial conditions for which \eqref{eq: ideal dimer} converges to the saddle point $S_1$ actually covers the region where there is a single negative eigenvalue of the Hessian of $V$. To obtain an implementable algorithm, we simply perform an explicit Euler discretization of \eqref{eq: ideal dimer}. Notice that there are potentially better algorithms than this simple explicit discretization of \eqref{eq: ideal dimer} for the local search step, see e.g. \cite{zhang2016optimization,barkema2001activation,gould2016dimer}. 

We are now in position to introduce the Stochastic Saddle Point Synamics with Local Search, see Algorithm~\ref{alg:capl}. We use the notation $\mathrm X=(X^1, \ldots, X^N)$ to denote the position of the $N$ particles in a single vector $\mathrm X \in \R^{Nd}$. We use
the set $\mathrm S(\mathrm X)=\{n~|~ X^n\in S \} \subset \{1, \ldots, N\}$
to denote all the indices of the particles which are at a point where the Hessian of $V$ has exactly one negative eigenvalue (see the definition~\eqref{eq:define set S} of $S$). If $\mathrm S(\mathrm X)$ is empty then we do not perform a local search. If $\mathrm S(\mathrm X)$ has more than one particle then we pick the one with the highest weight. Algorithm \ref{alg:capl} describes the procedure in full.
The main difference with Algorithm~\ref{alg:cap} is on lines 3-4 where every $m$ iterations, we look for a saddle point by launching the dimer search from a promising initial condition chosen among the particles, namely one with the largest weight among those which are in the index-1 region~$S$. 

For completeness we also specify the dimer algorithm in Algorithm \ref{alg:dimer}.
There are two reasons why we only consider points in the index-1 region $S$ to initialize the dimer, and why we stop the dimer algorithm when it leaves $S$: (i) according to~\cite[Theorem 2.2]{levitt2017convergence}, when $u(t)$ leaves the index-1 region, there is no guarantee of convergence to a saddle point and (ii) outside an index-1 region, the dimer method may converge to points which are not saddle points, see for example the point denoted by $S_1$ on the example~\cite[Figure 2]{levitt2017convergence}.

Notice that there are variants of the dimer algorithm~\cite{gao2015iterative,gould2016dimer} where it is even not required  to form the Hessian of $V$ or compute its eigenvalues. On the numerical examples we consider in this work, we did not have to implement these variants, but they may be useful in particular for very high-dimensional problems.

\begin{algorithm}
\caption{Stochastic Saddle Point Dynamics with Local Search (SSPD-LS)}\label{alg:capl}
\begin{algorithmic}[1]
\Require $K>1$, $\delta >0$, $m\geq 1$, $N\geq 1$, $\mathrm X_0$, $\mathrm Y_0$, $\rho_{ess}\in(0,1)$ 
\State $I=\emptyset$
\For{$k=0,\ldots,K$}
\If{$k \mod m =0$ and  $\mathrm S(\mathrm X_k)\neq \emptyset$}
\hspace{4cm} \  \text{\Comment{\tt Local Search}}
\State Add to $I$ the output of Algorithm~\ref{alg:dimer} initialized with $u_0=X^{l}_k$ \newline
\hspace*{1.1cm}  where 
$l=\argmax\limits_{n=1,\ldots,N}\{w^n_k~|~ n\in \mathrm S(\mathrm X_k) \}$, if this output is not already in $I$
\EndIf
\State $w^n_{k}=\|Y^n_{k}\|$ \hspace{5.8cm} $n=1,\ldots,N$ 
\If{$\text{ESS}(\mathrm w_{k})\leq \rho_{ess} N$}
 \hspace{3.5cm}\text{\Comment{Run resampling algorithm}} 
\begin{align*}
&J\leftarrow {\tt resample}(\mathrm w_k)& \ & \  \ \ \  \text{\Comment{Resample particles}} \\
 & X^n_k\leftarrow X_k^{J[n]}& \ & \  \ \ \ n=1,\ldots,N \\
& Y^n_k \gets Y^{J[n]}_k/w^{J[n]}_k & \ & \  \ \ \ n=1,\ldots,N 
 \end{align*}
\EndIf 
\For{$n=1,\ldots,N$} 
\State 
$X^n_{k+1}= X^n_{k}-\delta \, \nabla V (X^n_{k}) +\sqrt{2\beta^{-1} \delta} \, G^n_k$ \hspace{0.8cm}\text{\Comment{Update States}}
\State $Y^n_{k+1}= Y^n_{k}-\delta \, \nabla^2 V (X^n_{k}) Y^n_k $ 
\EndFor
\EndFor
\State \Return $I$
\end{algorithmic}
\end{algorithm}

\begin{algorithm}
\caption{Dimer Algorithm}\label{alg:dimer}
\begin{algorithmic}[1]
\Require $u_0$, $\delta>0$, $K_d>1$, $\epsilon_d>0$
\For{$k=0,\ldots,K_d$}
 \If {$\lambda_1(u_k)<0<\lambda_2(u_k)$ and $\|\nabla V(u_{k})\|<\epsilon_d$}  \Return $u_k$ 
\Comment{Success}
\ElsIf {$\lambda_1(u_k)>0$} \Return $\emptyset$
\Comment{Failure}
\Else \ $u_{k+1}=u_k-\delta (I_d-2 v_1(u_k) v_1(u_k)^\top)\nabla V(u_k)$
 \EndIf
\EndFor
\State \Return $\emptyset$ \Comment{Failure}
\end{algorithmic}
\end{algorithm}

\subsection{The particle dimer algorithm}\label{sec:algo1_pdimer}

Notice that Algorithm \ref{alg:capl} actually generates potentially more than one promising initial condition for the dimer algorithm. 
Indeed, at any iteration~$k$ such that $k \ \textrm{mod} \  m =0$, one can consider as initial conditions more than one particle in $\mathrm S(\mathrm X_k)$.

The conventional dimer method only evolves a single particle, and in  Algorithm \ref{alg:capl}, we only use the particle with the highest weight as an initial condition for the dimer method.
With this approach,  the algorithm tends to find a single saddle point associated with the particle with the largest weight, whereas one often wants to find a large number of saddle points, to get a complete picture of the energy landscape. To overcome this issue, one idea would be to lower the weights of particles which lie in basins of attraction (for the dimer dynamics) of
the previously visited index-1 saddle points. But it appears that those basins of attractions are difficult to identify, especially in high dimension.

To address this issue, we propose the following particle variant of the dimer dynamics on $t \mapsto \mathrm u(t) \in \mathbb R^{Md}$, where $M$ is a positive integer:
\BE\label{eq:dimer concensus}
\dot{\mathrm u}=-(I_{Md}-2\mathrm v_1(\mathrm u)\mathrm v_1(\mathrm u)^\top)\nabla \mathrm V(\mathrm u)
- \Lambda (\mathrm u_0 ) \mathrm u,
\EE 
with initial condition $\mathrm u(0)=\mathrm u_0$.
For a given $\mathrm X \in \mathbb R^{Nd}$ and associated weights $\mathrm w \in \R^N$, the initial condition $u_0$ is provided by stacking the $M$ particles in $\mathrm S(\mathrm X)$ with the largest weights into an $\R^{Md}$ vector. Here, $M$ is an additional parameter fixing the maximum number of initial conditions (of course, if $M$ is larger than the cardinality of $\mathrm S(\mathrm X_k)$, then one simply uses all the particles in $\mathrm S(\mathrm X_k)$ as initial conditions). The following notation is used in~\eqref{eq:dimer concensus}: $\mathrm u=[(u^1)^\top,\ldots,(u^M)^\top]^\top$,
$\mathrm v_1(\mathrm u)=[v_1(u^i)^\top,\ldots,v_1(u^M)^\top ]^\top$ and 
$\nabla{\mathrm V}(\mathrm u)=[\nabla V(u^1)^\top,\ldots,\nabla V(u^M)^\top]^\top$. The matrix $\Lambda(\mathrm u_0 )\in\R^{Md\times Md}$ is the Laplacian matrix associated with a graph $\mathcal G=(\mathcal V,\mathcal E)$ whose construction we describe below.

The vertex set $\mathcal V$ of the graph are the $M$ particles $(u_0^m)_{1\le m \le M}$. For $1 \le i,j \le M$, the $(i,j)^{\text{th}}$ entry of the weighted adjacency matrix $W$ is given by, for some positive cut off parameter $r>0$,
\BE\label{eq:kernel cuttoff}
W_{ij}=
K(u^0_i,u^0_j) 1_{K(u^0_i,u^0_j)<r } 
\EE
where $K:\R^d\times\R^d\rightarrow \R_+$ is a kernel designed to measure the distance between particle~$i$ and particle~$j$. If the parameter $r$ is chosen to be small, then only particles that are very close will interact. In the extreme case when $r=0$, the dynamics in \eqref{eq:dimer concensus} reduce to $M$ independent runs of Algorithm \ref{alg:dimer} with different initial conditions.
The disadvantage of choosing a small $r$ is that many particles tend to visit areas that may be far away from any saddle point and hence running $M$ independent runs of Algorithm \ref{alg:dimer} means that the time to convergence is determined by the worst initial condition.   
On the other hand if a very large~$r$ is chosen, then all the particles will interact:
if the particles are in the neighborhood of different saddles then this may prevent the algorithm to converge.  In our numerical experiments we used the Gaussian kernel
$K(x,y)=\exp(-\|x-y\|^2/\sigma)$ (see Section \ref{sec:numerical} for more details). While the choice of the kernel and  of the parameter $r$ parameter may have in principle a significant impact on the performance of the algorithm, we found that the algorithm is not very sensitive to this choice over a relatively large range of values of $\sigma$ and $r$ in the numerical tests we considered in Section~\ref{sec:numerical}. For this graph $\mathcal G$ and weighted adjacency matrix $W$,
the graph-Laplacian matrix is defined by
\[
\Lambda (\mathrm u_0)= I_{d}\otimes(D-W),
\]
where $D=\text{diag}(d_1,\ldots,d_M$), with $d_i=\sum\limits_{j=1}^M W_{ij}$. For $1 \le i \le M$, the $i$-th element of $\Lambda (\mathrm u_0)u$ is a vector in $\mathbb R^d$ defined by $\sum_{j=1}^M W_{ij} (u_i-u_j)$. We note that, by construction,  each row of $\Lambda$ sums to $0$. In addition $\Lambda u=0$ if and only if $u^i=u^j$, for $(i,j)\in \mathcal E$. The construction of such a graph-Laplacian matrix above is standard (see e.g. \cite{motsch2014heterophilious}). 

The obtained variant of the dimer algorithm is summarized in Algorithm~\ref{alg:particle dimer} and called in the following the particle dimer algorithm. It will be used in the local search step of Algorithm~\ref{alg:capl} (see lines 3-4) in place of the previously introduced dimer algorithm (Algorithm~\ref{alg:dimer}), the initial condition $\mathrm u_0 \in \mathbb R^{M_k d}$ being the $M_k$ particles in $\mathrm S(\mathrm X_k)$ with the $M_k$ largest weights, where $M_k=\min(M,|\mathrm S(\mathrm X_k)|$) ($M>0$ is thus the maximum number of particles considered in the particle dimer algorithm).

\begin{remark} The particle dimer dynamics \eqref{eq:dimer concensus} are merely extensions of the idealized dimer dynamics \eqref{eq: ideal dimer} with the additional linear term $\Lambda(u_0)u$ enforcing consensus, in the spirit of consensus based optimization algorithms. 
We derived the particle dimer algorithm drawing inspiration from the decentralized gradient descent method \cite{yuan2016convergence}. \panosb{The objective of the consensus term is to avoid  slow convergence because of some initial conditions for which the dimer algorithm does not converge quickly. Notice however that this consensus term only introduces local interactions between particles, so that the particle dimer is indeed able to locate many saddle points, when the initial conditions are  in the neighborhoods of different saddle points.
The variety in the initial points is effectively managed through the sampling mechanism outlined in Algorithm \ref{alg:capl}.}

\end{remark}

\begin{algorithm}
\caption{Particle Dimer Algorithm}\label{alg:particle dimer}
\begin{algorithmic}[1]
\Require $\mathrm u_0$, $\delta>0$, $K_d>1$, $\epsilon_d>0$,  $M>0$
\For{$k=0,\ldots,K_d$}
 \If {$\forall \ m=1\ldots,M$,\,  $\lambda_1(u^m_k)<0<\lambda_2(u^m_k)$ and $\|\nabla V(u^m_{k})\|<\epsilon_d$ } 
\State \Return $\mathrm u_k$ \Comment{Success}
\ElsIf {${\min\limits_{m=1\ldots,M} \lambda_1(u^m_k)}>0$} \Return $\emptyset$ \Comment{Failure}
 \Else \ $\mathrm u_{k+1}=\mathrm u_k-\delta(I_{Md}-2 \mathrm v_1(\mathrm u_k) \mathrm v_1(\mathrm u_k)^\top)\nabla \mathrm V(\mathrm u_k)
 -\delta \Lambda (\mathrm u_0)\mathrm u_k $
  \EndIf
\EndFor
\State  \Return $\emptyset$ \Comment{Failure}
\end{algorithmic}
\end{algorithm}

\section{Numerical experiments}\label{sec:numerical}
In this section, we report on numerical experiments performed on several benchmark problems. 
The results presented in Sections \ref{sec:double well}, \ref{sec: muller brown}, and \ref{sec:challenge} essentially illustrate that the proposed algorithms behave as expected from the theoretical derivations above. 
To support our claims, we run Algorithm \ref{alg:cap} and its variant Algorithm \ref{alg:capl} (where the local search is either performed by Algorithm \ref{alg:dimer} or \ref{alg:particle dimer}) on simple 2d examples. 
These experiments essentially aim at demonstrating the proposed methodology visually. 
Section \ref{sec:vac diffusion} then reports on results from a standard problem in molecular physics (vacancy diffusion) and illustrates that the computational cost of the algorithm increase linearly with the dimension of the configuration space. 
Finally, in Section \ref{sec:lj7}, we present results obtained on the 7-atom Lennard Jones potential in 2d.

We implemented\footnote{The code is available from the web-page of the second author.} the algorithms in Python 3.1 using the JAX-MD library~\cite{schoenholz2021jax}. 
JAX-MD is a Python library that provides primitive functions, operations and automatic differentiation for physical simulations.
It also allows Just-In-Time compilation to CPU, GPU and TPU. 
Code written in JAX-MD tends to be slower than hand written code in~C or~Fortran.
However, the added flexibility, code readability and ease of implementation outweigh the small efficiency \panosb{overhead}. 
We performed the experiments on an e2-standard-16 machine on the Google cloud. These machines have an Intel 2.0 GHz processor with 28 cores, 16 vCPUs, and 64GBs memory. The code is available from the webpage of the second author.

We note that, when comparing the performances of Algorithm~\ref{alg:capl} combined with Algorithm~\ref{alg:dimer}, and  Algorithm~\ref{alg:capl} combined with Algorithm~\ref{alg:particle dimer}, we use the same seed so that the swarms of particles evolve in the same way in both cases.
Thus the results only differ in the local saddle point search initialized from this ensemble of particles. 

The main numerical parameters we used are provided in Table \ref{table:parameters}. 
For Algorithm \ref{alg:cap}, the most important parameters are $\delta$ (time discretization parameter used in the Euler-Maruyama discretization of~\eqref{eq:sde full}) and $\beta^{-1}$ (temperature).
 We observed that our results were not very sensitive to the temperature parameter $\beta^{-1}$ chosen in the range $[0.01,0.1]$ (see a discussion in Section~\ref{sec:challenge}): smaller temperatures typically slows down the dynamics, but leads to clouds of particles which are more concentrated on the saddle points. The upper bound on $\beta^{-1}$ is typically driven by physics: for too large temperatures, the molecular dynamics become unstable. The lower bound on $\beta^{-1}$ influences the performance of our algorithm: for too small temperatures, the selection mechanism is difficult to tune in order to counter-act the steepest decent dynamics. 

Notice that we used in Algorithm~\ref{alg:dimer} and Algorithm~\ref{alg:particle dimer} the same parameter $\delta$ to discretize in time the dynamics~\eqref{eq: ideal dimer} and~\eqref{eq:dimer concensus} as the one given in Table \ref{table:parameters}. The accuracy 
threshold $\epsilon_d$ chosen for Algorithm~\ref{alg:dimer} and Algorithm~\ref{alg:particle dimer} was $10^{-6}$. 
For Algorithm~\ref{alg:particle dimer} we used $M=N=2000$, except for the high-dimensional potentials from Section \ref{sec:vac diffusion} and Section \ref{sec:lj7} where we used $M=100$. 
We used $K=10000$ as the maximum number of iterations for Algorithm~\ref{alg:cap} and Algorithm~\ref{alg:capl}. 
The maximum number of iterations in Algorithm~\ref{alg:dimer} and Algorithm~\ref{alg:particle dimer} is fixed to $K_d=1000$.
Finally, in the construction of the graph-Laplacian matrix, we used a Gaussian kernel (with $\sigma=1$) and the cut-off parameter in \eqref{eq:kernel cuttoff} was set to $r=0.9$.

\begin{table}[]
\begin{center}
\begin{tabular}{|l|ccccc|}
\hline
Section (Potential) & $\delta$ & $\beta^{-1}$ & $\rho_{ess}$ & $m$ & $N$  \\ \hline
  \ref{sec:double well} (Double-Well, Example \ref{example: toy2d}) &  $10^{-4}$& 0.1& 0.99&  10& 5000 \\
   \ref{sec: muller brown} (M\"uller-Brown)   & $10^{-3}$& 0.05& 0.95&  10& 2000   \\
   \ref{sec:challenge}  (2D Potential, \eqref{eq: challenging potential})  &  $10^{-3}$& [0.01,0.1]& 0.95&  10& 2000   \\
    \ref{sec:vac diffusion} (Vacancy Diffusion)  &  $10^{-4}$& 0.1& 0.99&  100& 2000  \\
   \ref{sec:lj7}  (Lennard-Jones)   &  $10^{-4}$& 0.1& 0.99&  100& 2000  \\ \hline
\end{tabular}
\end{center}
\caption{\label{table:parameters}Main parameters for Algorithms~\ref{alg:cap} and~\ref{alg:capl}.}
\end{table}

\subsection{Double-well potential}\label{sec:double well}
In this section we use the double-well potential from Example~\ref{example: toy2d} to illustrate numerically how Algorithm~\ref{alg:cap} performs.
Figure \ref{fig:double well t1} represents the particles after $k=10$ steps. The color of one particle represents its weight (the weights are scaled to be between 0 and 1: blue represents low weights and red particles have high weights). The arrow in the figure represents the direction provided by the following weighted average,
\[
\overline{Y}_k = \frac{\sum_{n=1}^N w^n_k Y^n_k}{\|\sum_{n=1}^N w^n_k Y^n_k\|}.
\]
The origin for $\overline{Y}_k$ is given by,
 \[
\overline{X}_k = \frac{\sum_{n=1}^N w^n_k X^n_k}{\sum_{n=1}^N w^n_k}.
\]
The grey shaded region represents the region where one eigenvalue of the Hessian of $V$ is positive and the other one is negative. After \panosb{$k=200$} time steps we see in Figure~\ref{fig:double well t200} that the particles move towards the shaded grey area. As expected, the particles inside the grey area have higher weights than the particles outside this region.
Figure \ref{fig:double well t400} shows the state of the algorithm after  \panosb{$k=400$} iterations. Clearly, a stationary state has been attained for Algorithm~\ref{alg:cap}, and the particles are confined to a region around the saddle point ($\overline{X}_{400}$ is approximately the saddle point).

\begin{table}
\begin{center}
\begin{tabular}{ |c|c|c|c| } 
 \hline 
 Label & Energy & $x_1$ & $x_2$ \\ 
  \hline
$C_1$ &-146.700 & -0.558&1.442\\ 
$C_2$ &-80.768  & -0.05 & 0.467 \\
$C_3$ & -108.167& 0.623 & 0.028\\ 
$T_1$ & -40.665 & -0.822& 0.624 \\
$T_2$ & -72.249 & 0.212 & 0.293 \\
 \hline
\end{tabular}
\end{center}
\caption{\label{table: muller potential}Minima and saddle points of the M\"uller-Brown potential.}
\end{table}
\begin{figure}
\centering     
\subfigure[$k=10$]{\label{fig:double well t1}\includegraphics[width=0.325\textwidth]{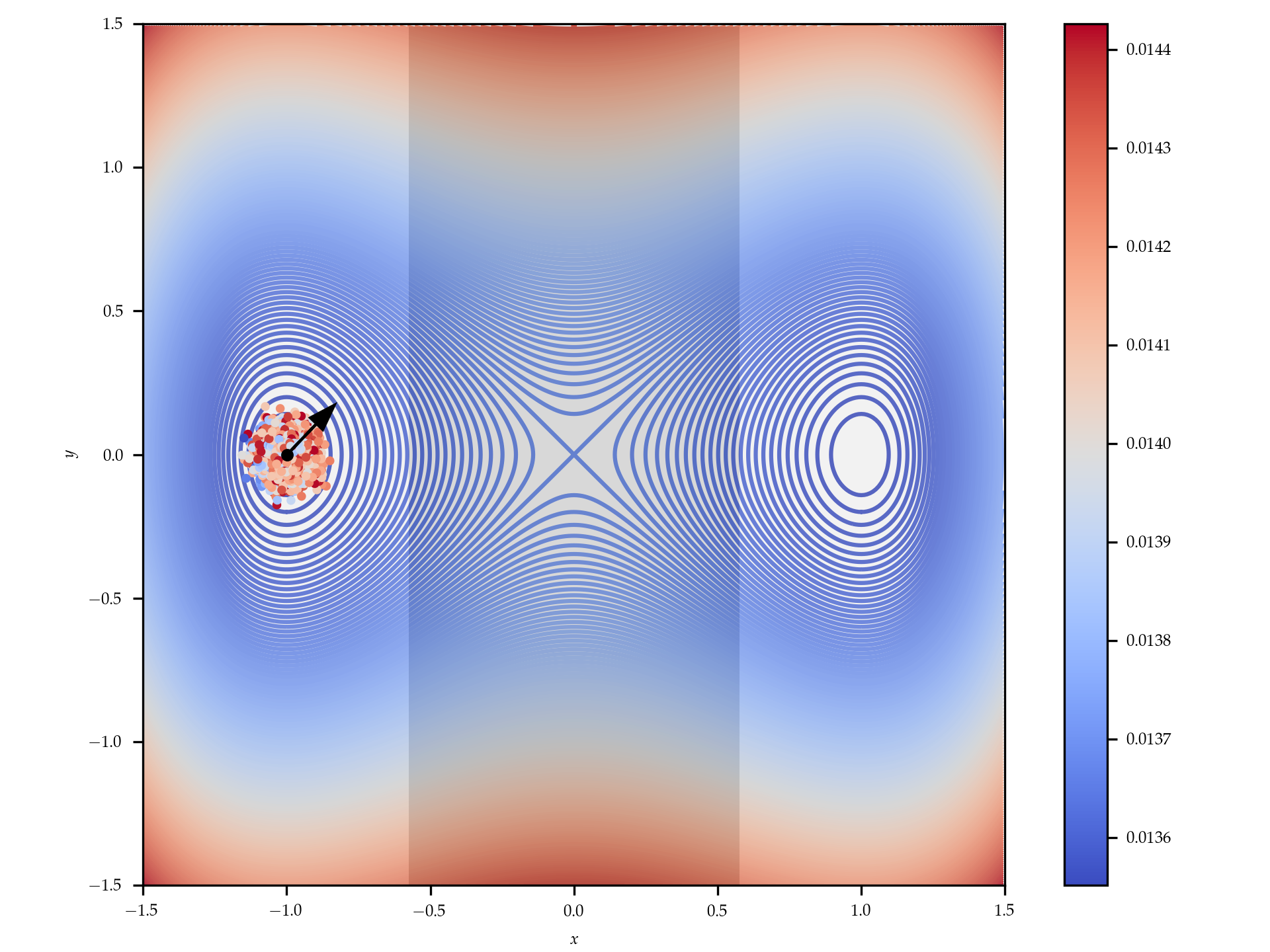}}
\subfigure[$k=200$]{\label{fig:double well t200}\includegraphics[width=0.325\textwidth]{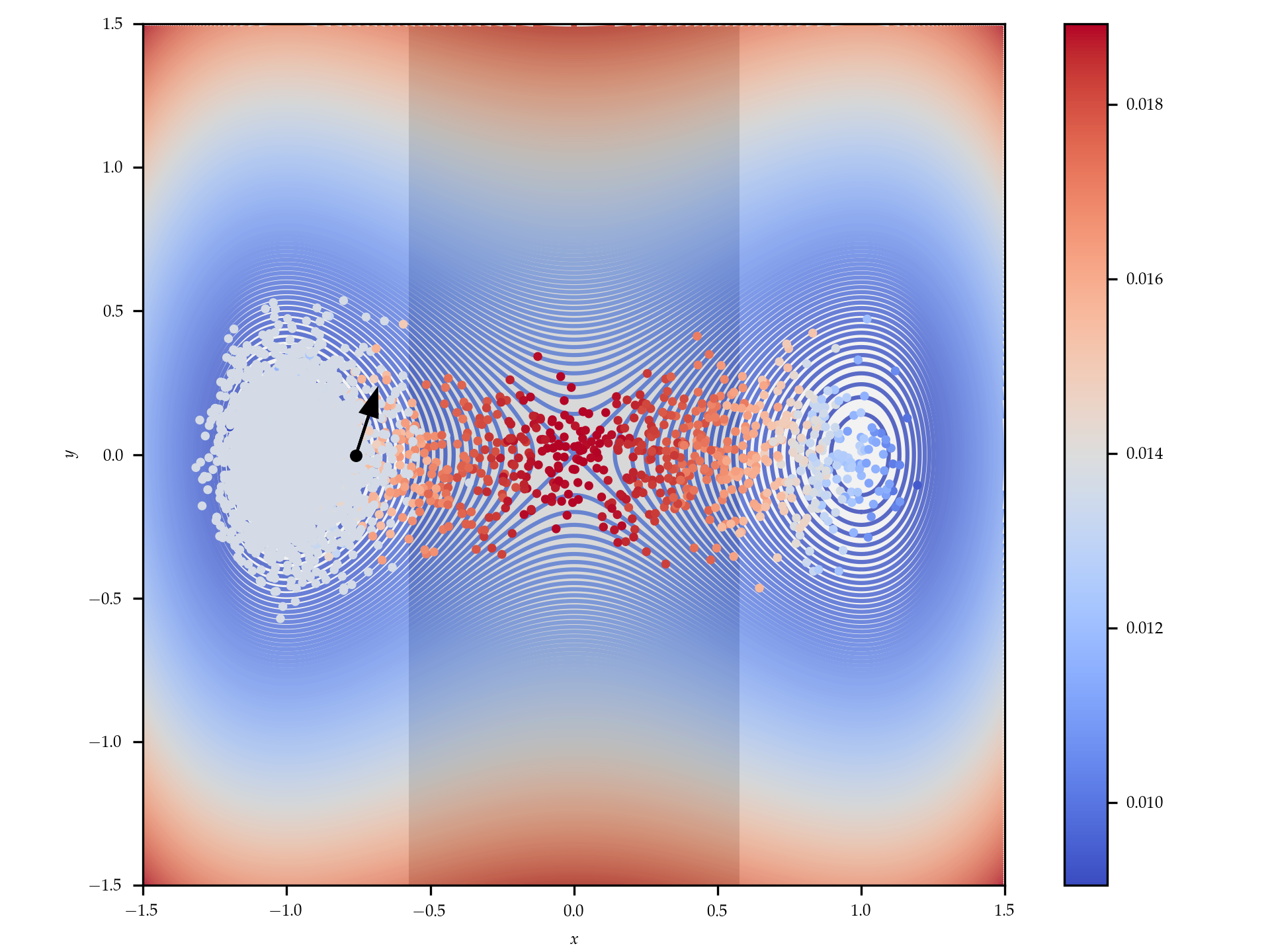}}
\subfigure[$k=400$]{\label{fig:double well t400}\includegraphics[width=0.325\textwidth]{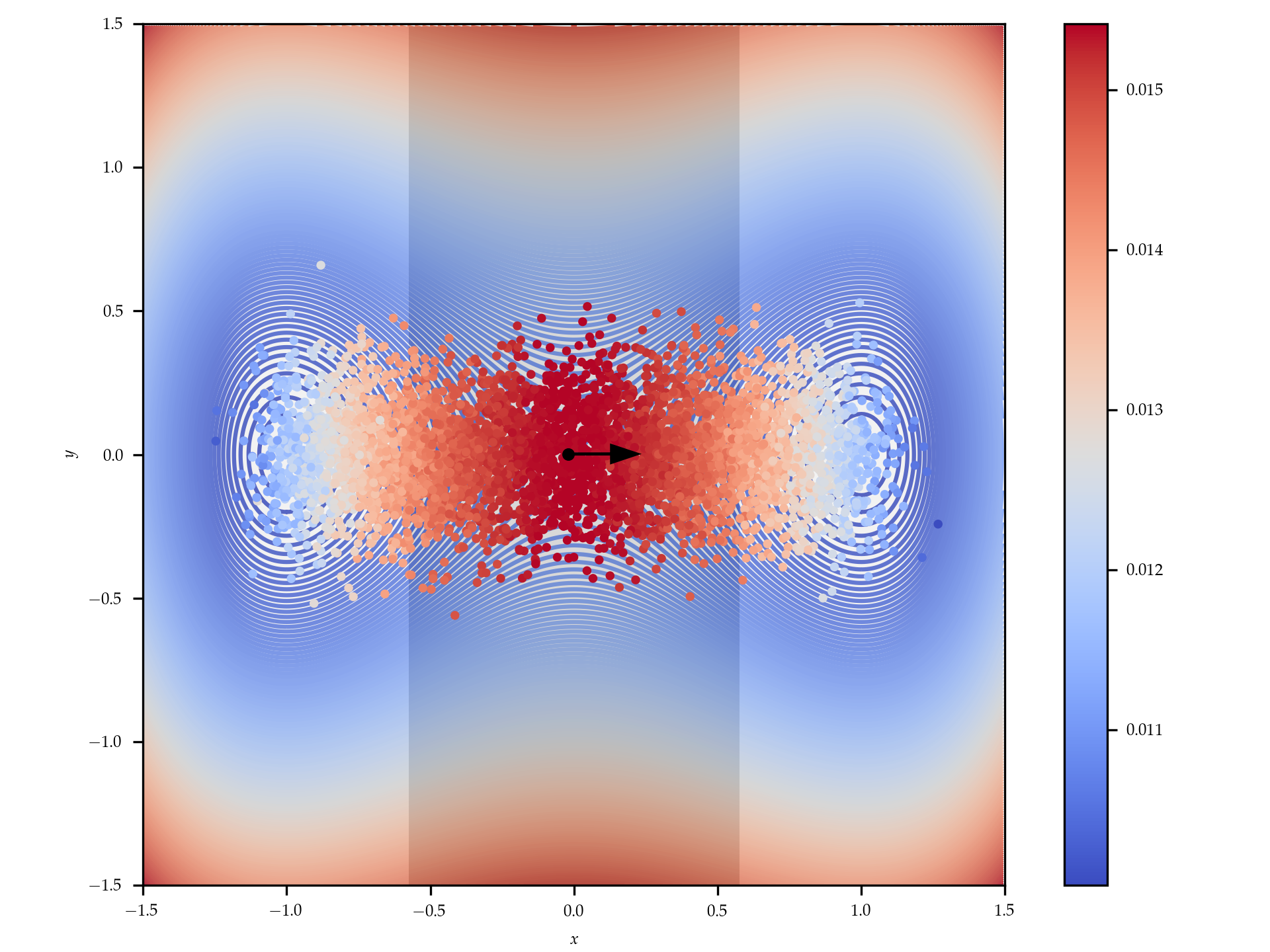}}
\caption{Double-Well Potential. An animation associated with this Figure is available in the supplementary material.}
\label{fig:double well animation}
\end{figure}
\begin{figure}
\centering     
\subfigure[M\"uller-Brown Potential]{\label{fig:muller brown}\includegraphics[width=0.295\textwidth]{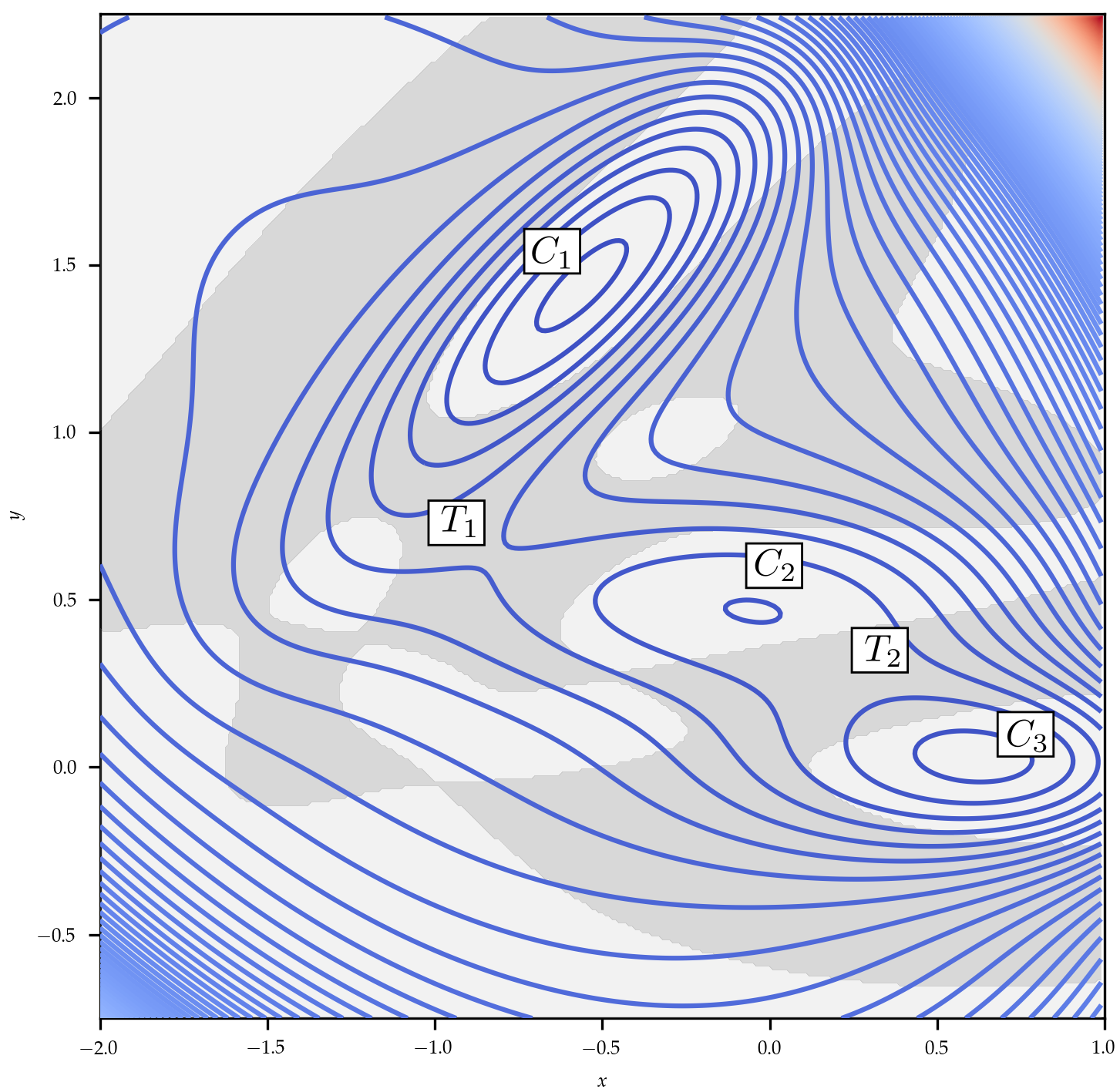}}
\subfigure[$k=1000$, local search with Algorithm \ref{alg:dimer}  ]{\label{fig:muller brown vanilla}\includegraphics[width=0.34\textwidth]{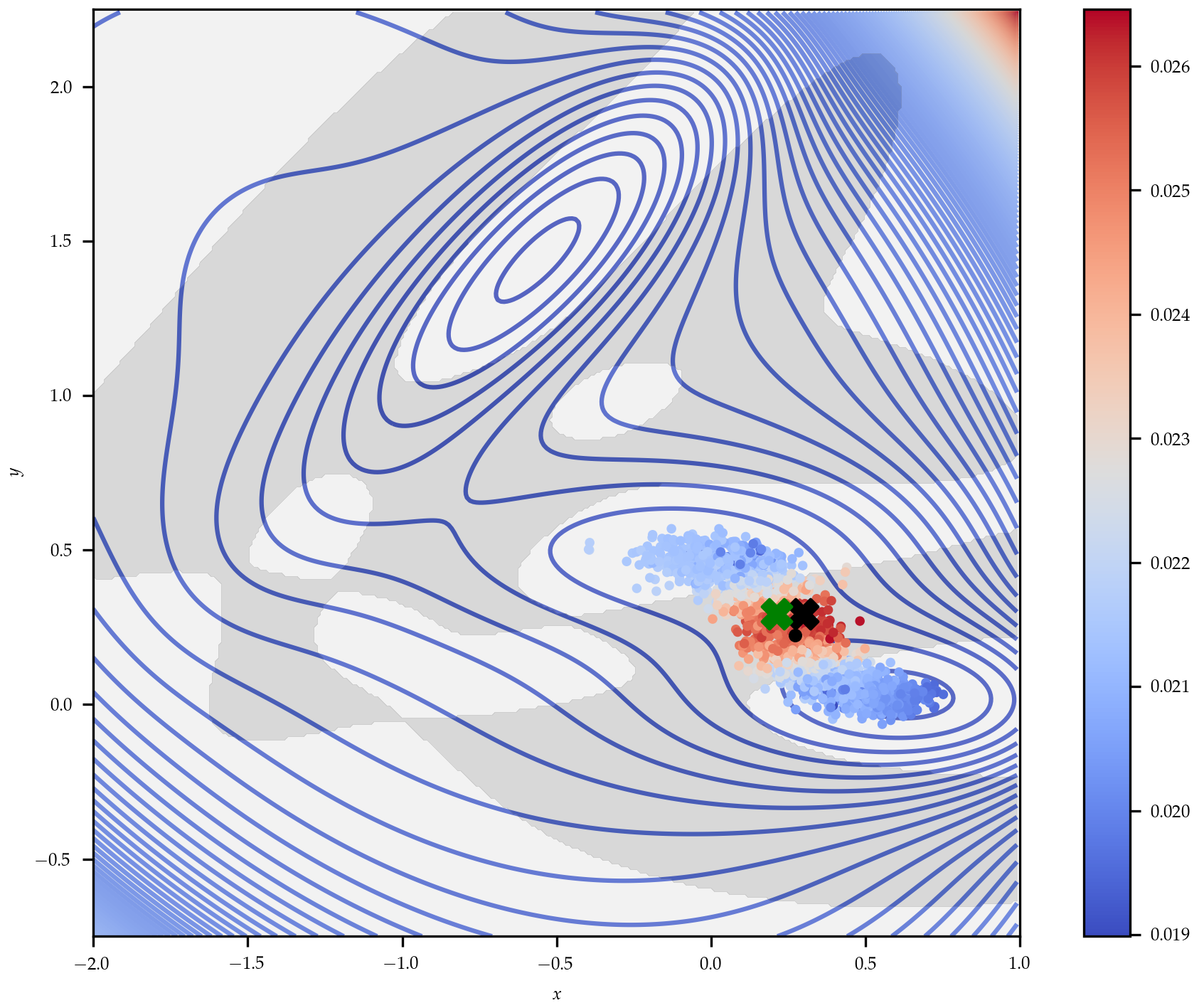}}
\subfigure[$k=10$, local search with Algorithm \ref{alg:particle dimer}]{\label{fig:muller brown particle}\includegraphics[width=0.34\textwidth]{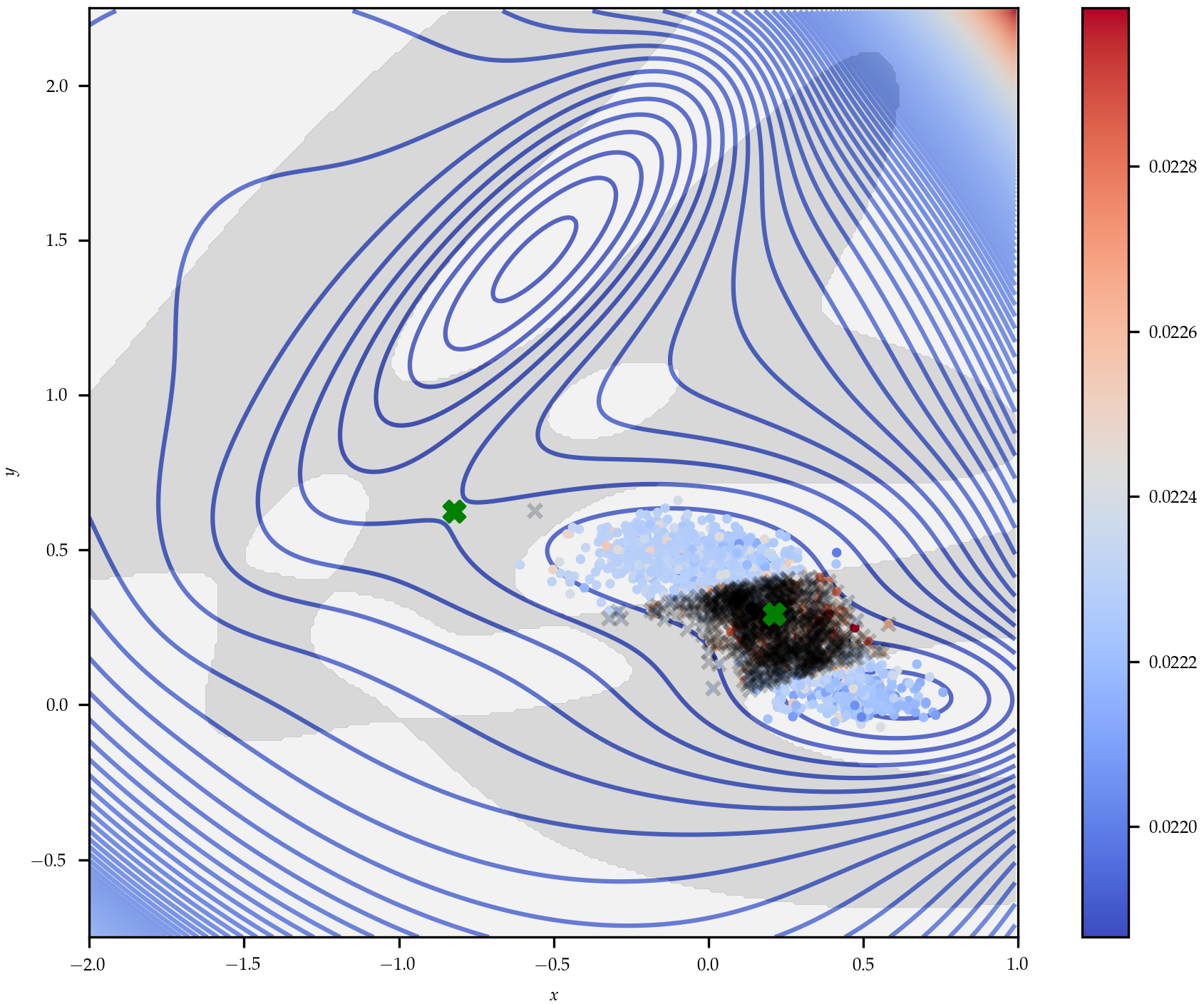}}
\caption{Experiments with the M\"uller-Brown Potential. }
\label{fig:muller-brown-experiment}
\end{figure}

\subsection{M\"uller-Brown potential}\label{sec: muller brown}
In this section, we report on experiments with the M\"uller-Brown potential. We use the same parameters as the ones reported on  \cite[pp.~298]{wales2003energy}. This potential is often used to illustrate the numerical performance of saddle-point algorithms. 
It has a more challenging geometry than the simple double-well potential of the previous section. 
Throughout this section, we use the same color scheme as in the previous section.
The M\"uller-Brown Potential has three local minima ($C_1$, $C_2$, $C_3$) and two saddle points ($T_1$, $T_2$).
These critical points are shown on Figure~\ref{fig:muller brown}, see also Table~\ref{table: muller potential}.

We start Algorithm \ref{alg:capl} at $C_2$ and then call Algorithm \ref{alg:dimer} every $m=10$ steps. 
In Figure~\ref{fig:muller brown vanilla} we see that the particles, after $1000$ steps, tend to concentrate around the $T_2$ saddle point which is lower in energy than $T_1$. In fact, we run Algorithm \ref{alg:capl} together with Algorithm \ref{alg:dimer} for a long time (up to $k=10^5$) and the algorithm did not find $T_2$.
We then run Algorithm~\ref{alg:capl} (with the same initial conditions, parameters and random seed) but, every $10$ steps, instead of calling Algorithm~\ref{alg:dimer}, we call Algorithm~\ref{alg:particle dimer}.
In Figure \ref{fig:muller brown particle}, we plot using grey crosses all the particles used to construct the adjacency matrix of Algorithm \ref{alg:particle dimer} at iteration $k=10$ (i.e. at the first iteration a local search was performed).
In this case, we observe that the resulting graph has two connected components. The first one has many particles around~$T_2$. The second component of the graph has only a single particle  that is close to~$T_1$. We note that this point does not have the highest weight and therefore was not selected as an initial condition in the previous experiment combining  Algorithm \ref{alg:capl} and Algorithm \ref{alg:dimer}.
On the other hand, Algorithm \ref{alg:capl} combined with Algorithm~\ref{alg:particle dimer} was able to identify $T_1$ and $T_2$ after only $k=10$ iterations (see the green crosses on Figure~\ref{fig:muller brown particle}).

This experiment illustrates the fact that Algorithm \ref{alg:particle dimer} is better suited to situations where there are multiple saddle points connected to a particular local minimum.

\subsection{A challenging 2d potential}\label{sec:challenge}
In this section, we discuss the behavior of Algorithm \ref{alg:capl} on a challenging 2d potential where the set  $S$ (defined in \eqref{eq:define set S}) is disconnected, namely $S=S_1\cup S_2$ with $S_1\cap S_2=\emptyset$, and only~$S_2$ contains a saddle point. For the sampling-based method we propose here, the fact that the set $S$ is disconnected presents a challenge because when we initialize the algorithm in one of the sets, say $S_1$, it is difficult for particles to discover $S_2$ since points outside $S$ tend to have a small weight. Moreover, it is known that the dimer method does not behave well in  a situation where $S_1$ does not actually contain any saddle point~\cite{levitt2017convergence}. 
In particular when initialized in $S_1$, there is no guarantee that the dimer algorithm will leave $S_1$ and will converge to a saddle point. 
In the specific example considered in this section, see Figure~\ref{fig:2d disconnect}, if the dimer algorithm is not stopped when leaving $S_1$, it will actually converge to an attractive singularity which is not a saddle point, see the discussion in~\cite[Section 1.4]{levitt2017convergence}.

The potential we used is inspired from~\cite[Section 1.4]{levitt2017convergence}. 
The original potential from~\cite[Section 1.4]{levitt2017convergence} (see $V_1$ below) has an index-1 region but no saddle points. As explained above, the authors used it to illustrate a situation where the dimer method converges to a point that is not an index-1 saddle point. 
Our proposed approach should not suffer from this pathological behavior, since we only use the dimer method to speed up the search of the stochastic saddle point dynamics, which is guaranteed to concentrate in neighborhoods of saddle points\panosb{.} To illustrate this point, we modified the potential from~\cite[Section 1.4]{levitt2017convergence} by adding a local minimum and a saddle point in another index-1 region, sufficiently far away from $S_1$ in order not to perturb the pathological features of the original potential. The modified potential is,
\begin{equation}\label{eq: challenging potential}
V(x_1,x_2)=V_1(x_1,x_2)/Z-V_2(x_1,x_2),
\end{equation}
where $Z=4\times 10^{3}$, $
V_1(x_1,x_2)=(x_1^2 + x_2^2)^2 + x_1^2 - x_2^2 - x_1 + x_2$, and $V_2(x_1,x_2)=e^{-((x_1 - 5)^2 + (x_2 - 5)^2))}$. 
We plot this potential in Figure \ref{fig:2d disconnect} and indicate the two local minima by $M_1$ (the original minimum from~\cite{levitt2017convergence}), and $M_2$ (the local minimum we added). 
Like in earlier sections, the grey region indicates the set $S$, which here has two disjoint connected components $S_1$ and $S_2$. 
The two minima are connected by the saddle point $T_{1} \in S_2$. We initialize Algorithm \ref{alg:capl} at $M_1$. By choosing the temperature parameter $\beta^{-1}$ sufficiently large, our algorithm explores the space outside $S_1$ and, combined with the simple dimer search of Algorithm~\ref{alg:dimer}, it is able to locate $T_1$. 


\begin{figure}
\centering     
\subfigure[The 2d Potential in \eqref{eq: challenging potential} ]{\label{fig:2d disconnect}\includegraphics[width=0.3\textwidth]{{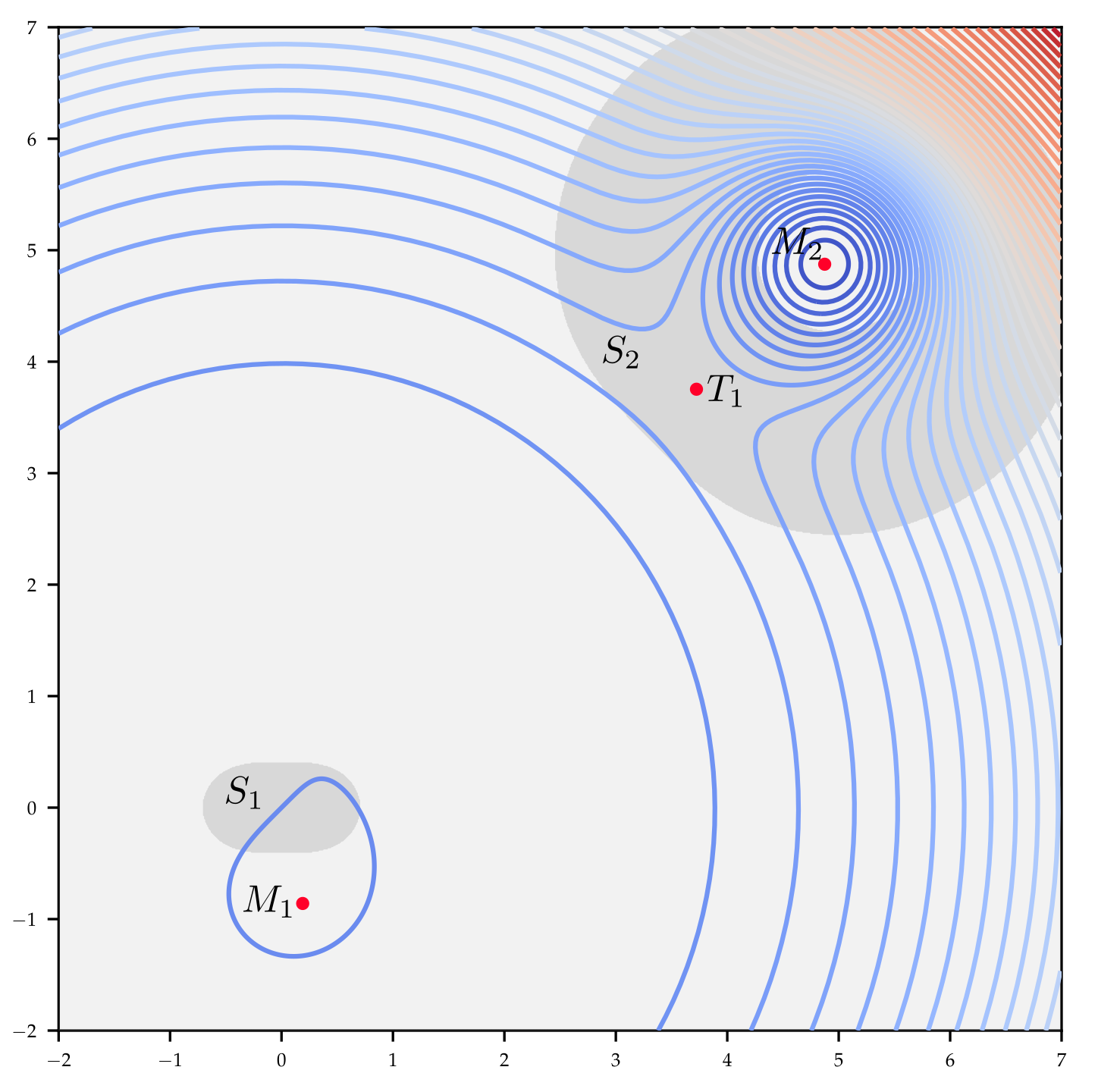}}}
\subfigure[\# Iterations to locate $T_1$ as a function of $\beta^{-1}$]{\label{fig:2d disconnect beta}\includegraphics[width=0.45\textwidth]
{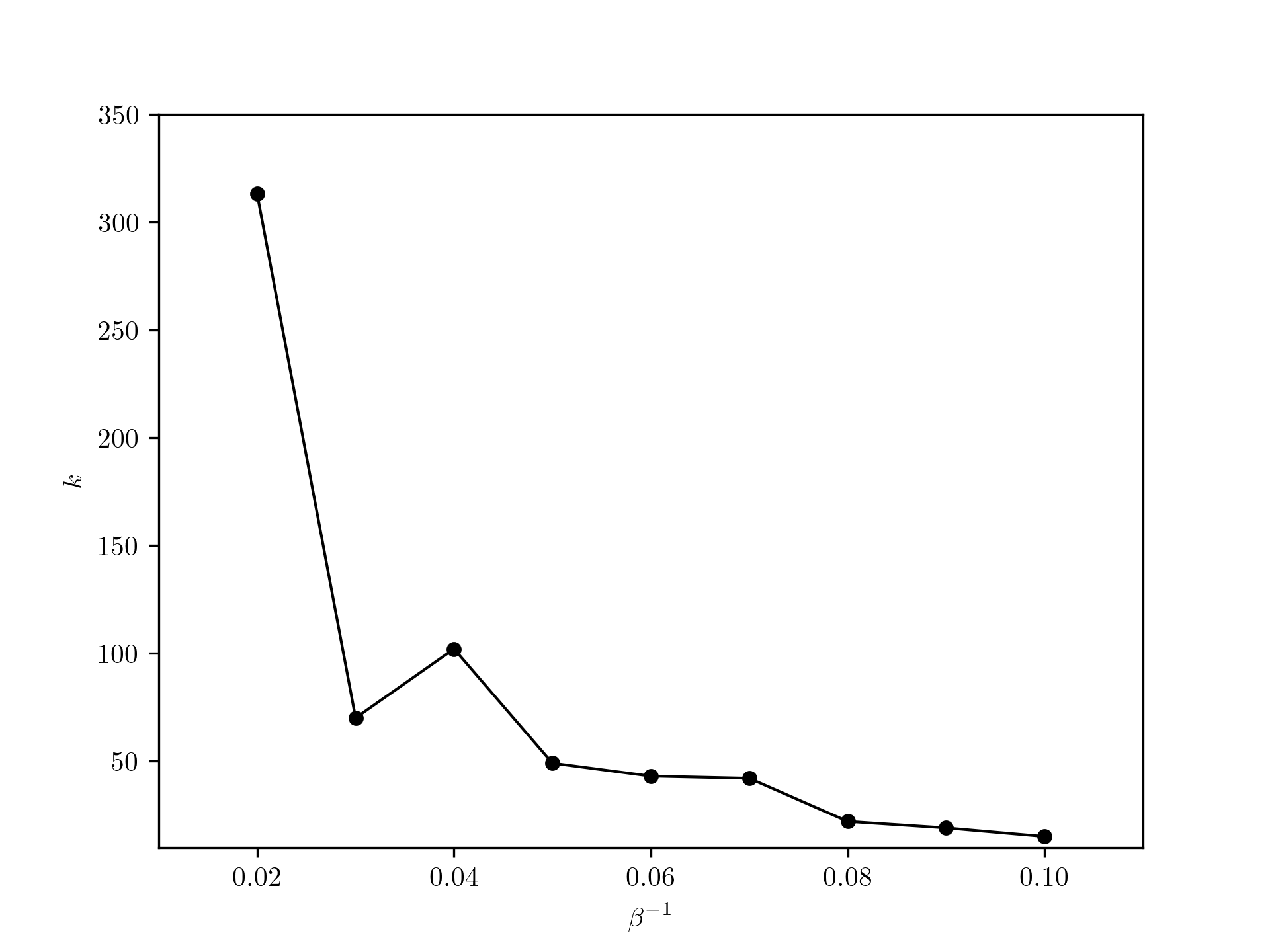}}
\caption{Experiments with a challenging potential. (a) $M_1$ and $M_2$ are local minima, $T_1$ is a saddle point, and $S=S_1\cup S_2$ with $S_1\cap S_2=\emptyset$.  
(b) The amount of iterations it takes for Algorithm \ref{alg:capl} to identify $T_1$ as a saddle point for Algorithm \ref{alg:dimer} as a function of $\beta^{-1}$. The algorithm was initialized from $M_1$\panosb{.}}
\label{fig:2d disconnect both}
\end{figure}

To quantify the regime where $\beta^{-1}$ becomes sufficiently large,  we plot in Figure \ref{fig:2d disconnect beta} the first iteration a call to Algorithm~\ref{alg:dimer} from Algorithm~\ref{alg:capl} is successful (we note that we only run a local search every $m=10$ iterations, so these are necessarily multiples of $10$). 
The lower the temperature, the longer it takes to identify a good initial condition. 
However, even at low temperature, the algorithm  eventually finds it: this just requires a larger number of iterations.

\subsection{Vacancy diffusion}\label{sec:vac diffusion}
So far, we have presented several numerical experiments suggesting the proposed method works as expected in 2d potentials.
In the last two sections, we would like to show that the algorithms we propose are able to locate saddle points on higher dimensional problems. Thanks to the stochastic representation \eqref{eq:sde full} of the Witten partial differential equation~\eqref{eq:FP1}, the algorithm can indeed be implemented in any dimension. In this section, we will discuss in particular how the algorithms can indeed scale to high dimensional problems. 

We use a standard benchmark problem from molecular physics, namely a vacancy diffusion model. 
The energy function is defined using the Morse potential as a pair potential. 
We used the same parameters for the potential as in \cite{gould2016dimer}. 
The atoms are arranged in a lattice, as shown in Figure~\ref{fig:vacancy diffusion}. 
The black \panosb{disks} in the lattice represent atoms that are fixed, whereas the red \panosb{disks} represent atoms that can move freely. 
Figure~\ref{fig:vacancy diffusion} illustrates the motion of a vacancy, an atom moving from its current position (shown in the grey square) to fill a neighbouring vacant position. The configuration in Figure~\ref{fig:vd_index_0} represents the local minimum we initialized Algorithm \ref{alg:capl} from. The structure in Figure~\ref{fig:vd_index_1} is one of the lowest saddle points for the vacancy to move (by the symmetry of the problem, there are of course six equivalent saddle points associated with six equivalent moves of the vacancy). 

For this problem, we run Algorithm~\ref{alg:capl} in conjunction with Algorithm~\ref{alg:dimer}. We initialize the algorithm from the local minimum configuration displayed on Figure~\ref{fig:vd_index_0}. It is easy to change the  dimensionality of the problem by increasing the number of atoms in the lattice that can freely move. It is, therefore, an ideal problem to see how the computational cost of Algorithm~\ref{alg:capl} scales with the dimension of the problem. 
\begin{figure}
\centering     
\subfigure[Local minimum configuration.]{\label{fig:vd_index_0}\includegraphics[width=0.4\textwidth]{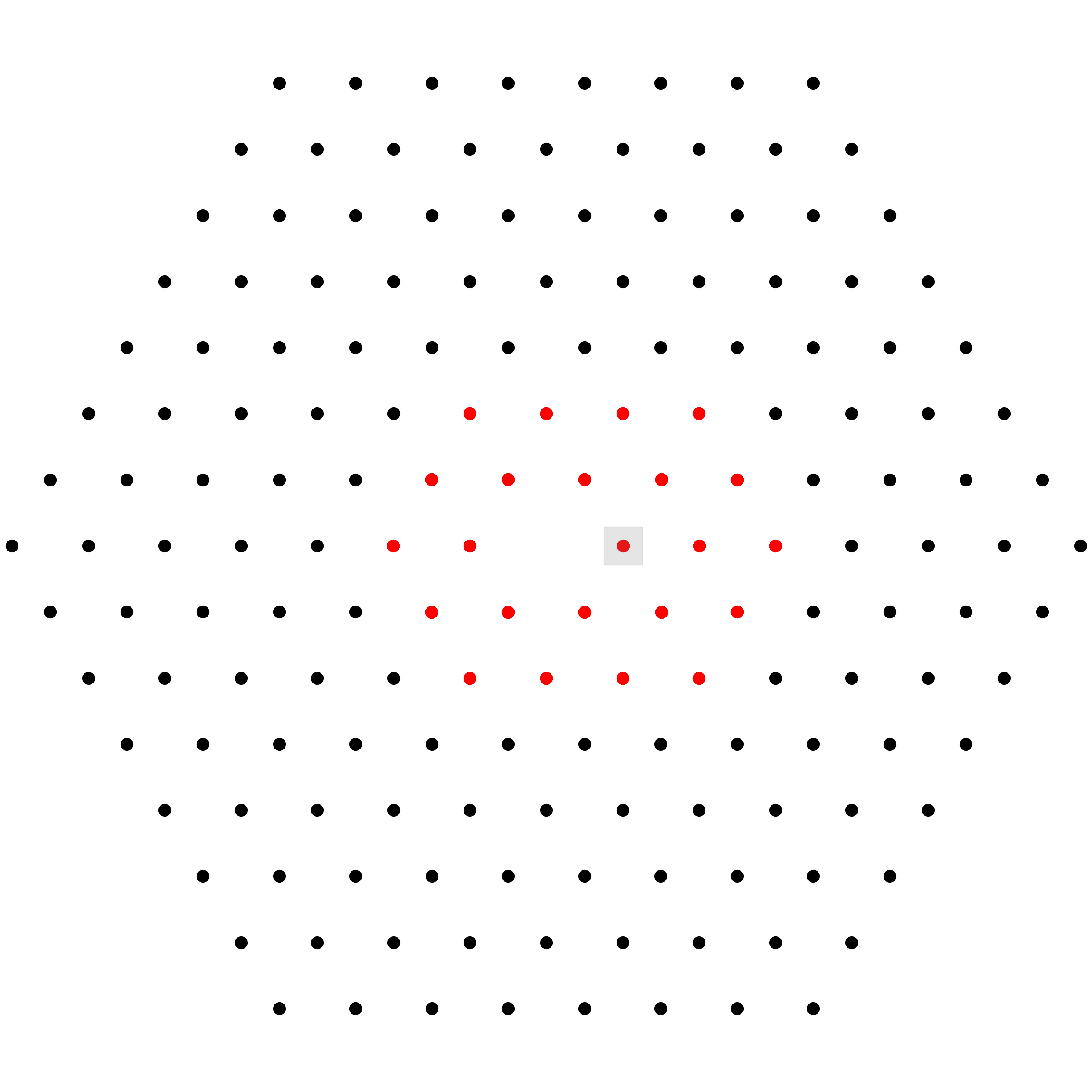}}
\subfigure[Saddle point configuration.]{\label{fig:vd_index_1}\includegraphics[width=0.4\textwidth]{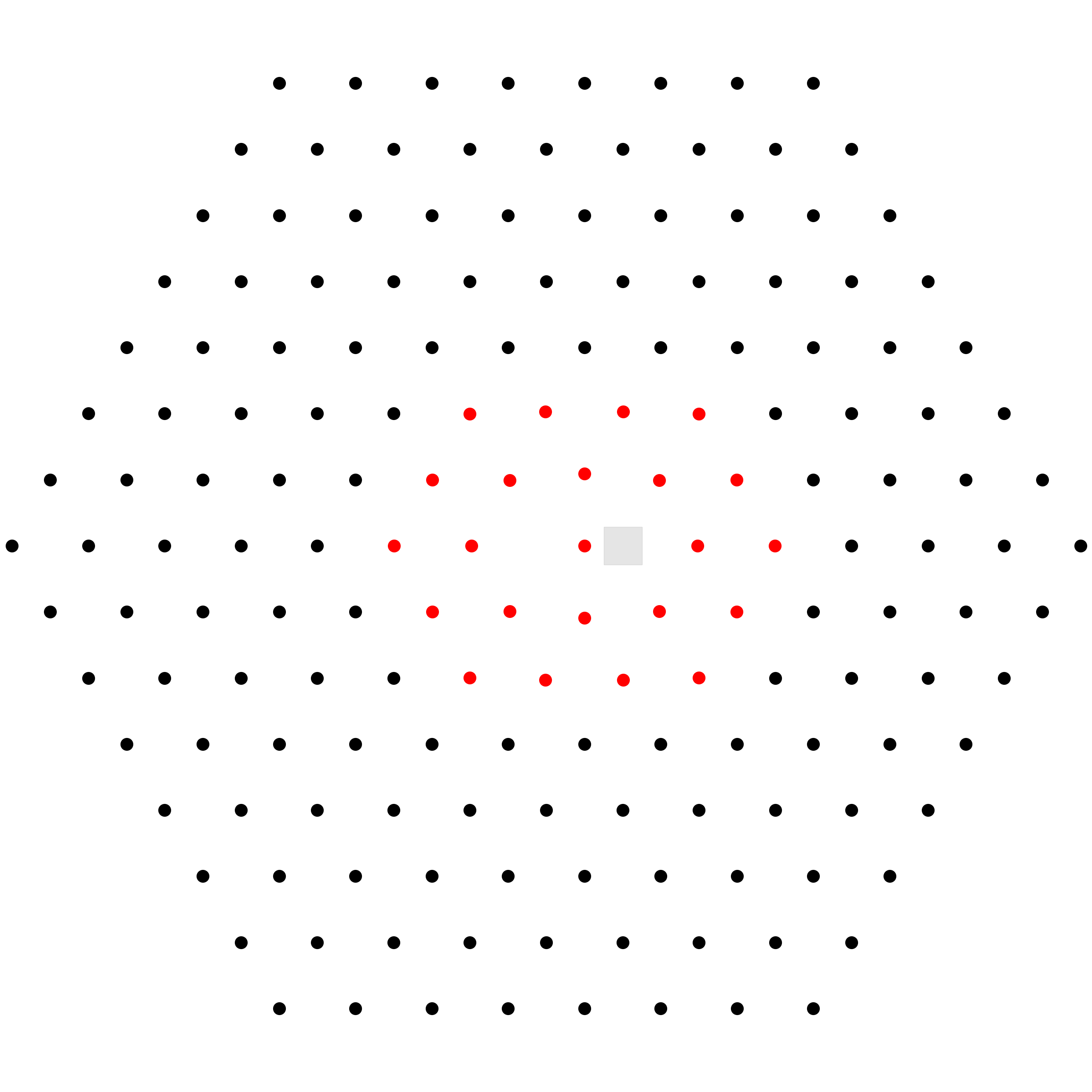}}
\caption{Lattice configurations for the vacancy diffusion problem with 23 free (red) and 145 fixed (black) atoms.}
\label{fig:vacancy diffusion}
\end{figure}
\begin{table}
\begin{center}
\begin{tabular}{ |c|c|c| } 
 \hline
 Dimension & CPU-Time (s)  & $k$  \\ 
  \hline
 18 & 0.85 &  500 \\ 
 46 & 1.68 & 700 \\
  138 & 3.84 & 300 \\ 
  202 & 5.59 & 1000 \\
  278 & 8.21 & 700 \\
 \hline
\end{tabular}
\end{center}
\caption{CPU-time to perform $m=100$ iterations of Algorithm~\ref{alg:capl} on the vacancy diffusion test case, as a function of the dimension of the problem. The third column indicates the first iteration $k$ where a saddle point was identified.}\label{tab:vacancy diffusion}
\end{table}

In Table~\ref{tab:vacancy diffusion}, we give the amount of CPU time (in seconds) it takes to perform $m=100$ steps of Algorithm~\ref{alg:capl} (without including the time of the local dimer search, which is a post-processing step). 
It is clear from Table~\ref{tab:vacancy diffusion} that the algorithm scales linearly with the dimension of the problem, which is twice the number of free atoms. 
There are two reasons for this favorable scaling. 
Firstly, the computations in Steps 8-9 of the algorithm involve only the calculations of gradients of~$V$ and of products of Hessians of~$V$ by vectors: using algorithmic differentiation, both calculations can be performed in time and space complexity that is linear in $d$. 
Secondly, Steps 8-9 can easily be performed in parallel. 
The only time the particles `exchange' information is through the re-sampling step, which has a negligible computational cost. 
Therefore, the structure of Algorithm \ref{alg:capl} fits well with modern parallel computing architectures where the costly operations are communications between cores and data movement. Table~\ref{tab:vacancy diffusion} also provides the first iteration $k$ when a saddle point is found, on a typical run. Notice that $k$ is by construction a multiple of $100$ since $m=100$. We observe that $k$ does not vary too much with the dimension.

\subsection{7-atom Lennard-Jones cluster in dimension 2}\label{sec:lj7}
The purpose of this section is to show that the proposed algorithm can be used to identify the  transition states (namely the saddle points) of a small molecular system, with a simple pairwise interacting potential. More precisely, we report results obtained for the 7-atom Lennard-Jones cluster in 2d, with potential:
\BE
\label{eq:lj7}
V(x_1,\ldots,x_7)=\sum_{i<j}4\epsilon
\left[
\left(\frac{\sigma}{\|x_i-x_j\|}\right)^{12}
-\left(\frac{\sigma}{\|x_i-x_j\|}\right)^6
\right],
\EE  
where $x_i\in\R^2$, $i=1,\ldots,7$ denote the coordinates of the $i^\text{th}$ particle. To keep the numerical values consistent with previous works, we set $\sigma=\epsilon=1$. This potential has been used in many other works and is a standard benchmark problem \cite{dellago1998efficient,wales2003energy,zhang2016optimization,forman2017modeling}. 

\overfullrule=0pt
\begin{table}[!htb]
     \begin{minipage}[t]{.33\linewidth}
      \caption{\label{tab:lj7-min}Minima}
  \begin{tabular}
      {|ccc|} \hline  
     Cluster & Label & Energy  \\
      \hline
      \parbox[c]{1em}{
\vspace*{0.1cm}
      \includegraphics[scale=0.5]{"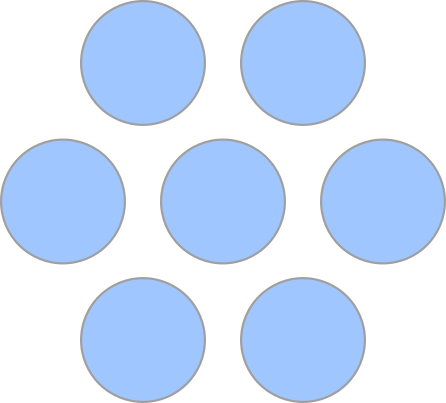"}} & $C_0$ & -12.535   \\  
        \parbox[c]{1em}{
   \vspace*{0.2cm}
      \includegraphics[scale=0.5]{"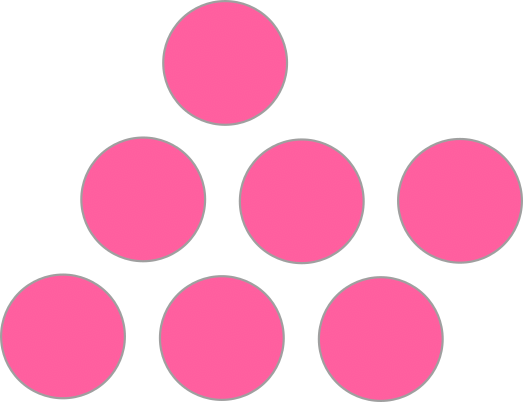"}} & $C_1$ & -11.501  \\  
      \parbox[c]{1em}{
      \vspace*{0.2cm}
      \includegraphics[scale=0.5]{"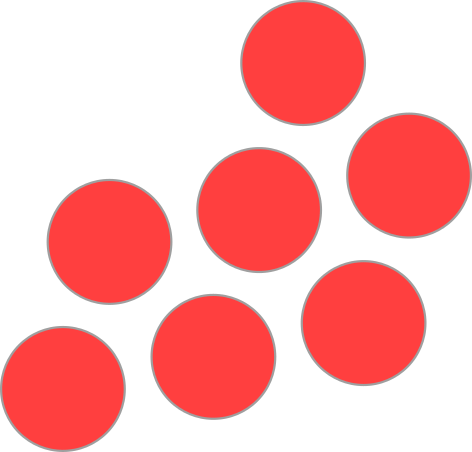"}
       } 
      & $C_2$ & -11.477   \\  

\parbox[c]{1em}
{
      \vspace*{0.1cm}
      \includegraphics[scale=0.5]{"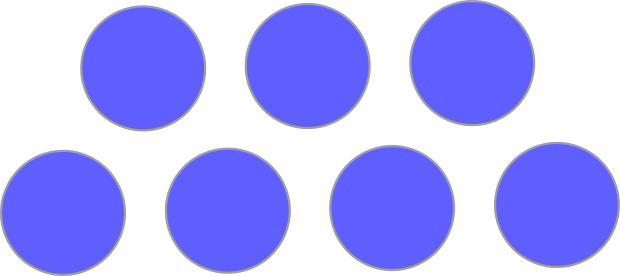"}} & $C_3$ & -11.403    \\
    \hline 
  \end{tabular}
\end{minipage}
    \begin{minipage}[t]{.6\linewidth}
      \caption{\label{tab:lj7-ts}Saddle points}
 \begin{tabular}
      {|ccc|ccc|} \hline  
     Cluster & Label & Energy & Cluster & Label & Energy   \\
      \hline

\parbox[c]{1em}{
\vspace*{0.1cm}
      \includegraphics[scale=0.5]{"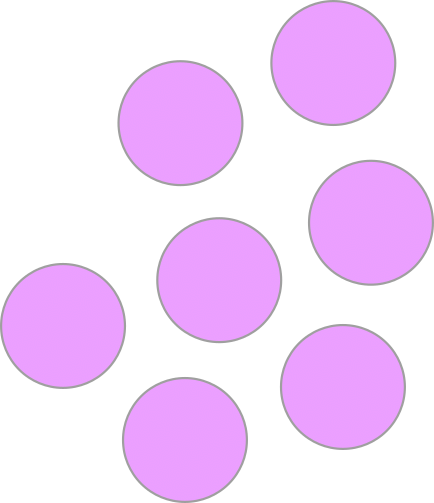"}} & $T^A_{0\leftrightarrow 2}$ & -11.040 &
      \parbox[c]{1em}{
\vspace*{0.1cm}
      \includegraphics[scale=0.5]{"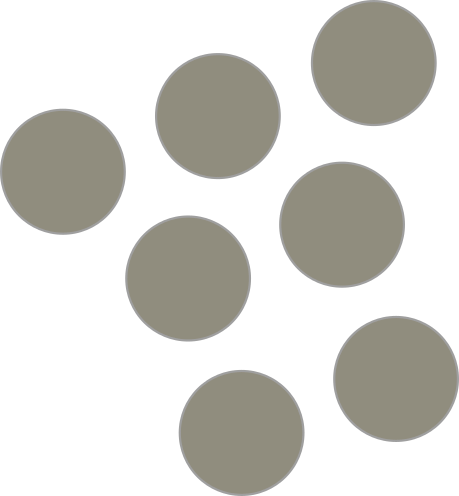"}} & $T^B_{1\leftrightarrow 1}$ & -10.841  \\ 
     
        \parbox[c]{1em}{
   \vspace*{0.1cm}
      \includegraphics[scale=0.5]{"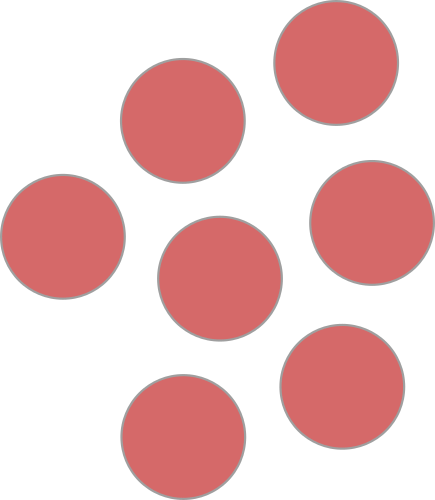"}} & $T^A_{1\leftrightarrow 0}$ &-11.037   &  
     \parbox[c]{1em}{
      \vspace*{0.1cm}
      \includegraphics[scale=0.5]{"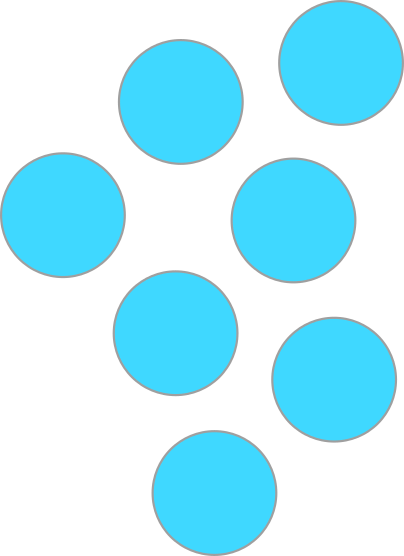"}} & $T^C_{1\leftrightarrow 2}$ & -10.807   \\  
\parbox[c]{1em}
{
      \vspace*{0.1cm}
      \includegraphics[scale=0.5]{"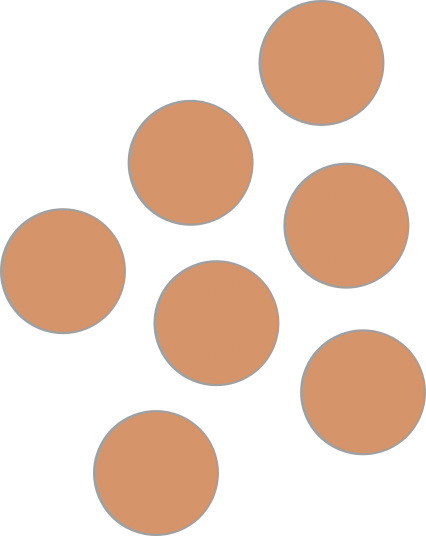"}} & $T^A_{1\leftrightarrow 1}$ &  -10.935  &

\parbox[c]{1em}
{
      \vspace*{0.1cm}
      \includegraphics[scale=0.5]{"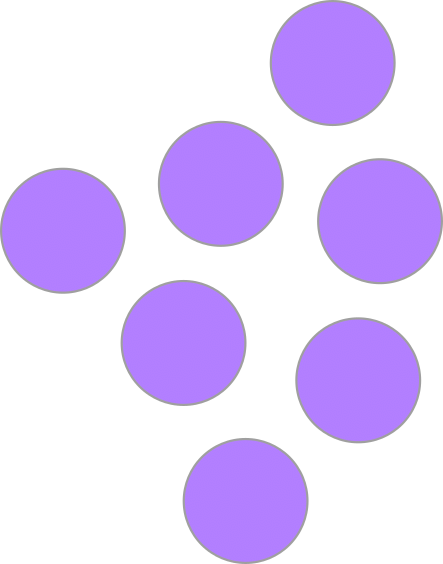"}} & $T^D_{1\leftrightarrow 2}$ & -10.799  \\
 \parbox[c]{1em}{
\vspace*{0.1cm}
      \includegraphics[scale=0.5]{"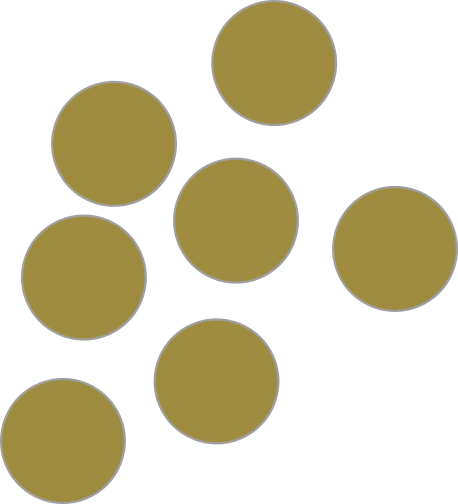"}} & $T^A_{1\leftrightarrow 2}$ & -10.919 &  
        \parbox[c]{1em}{
   \vspace*{0.1cm}
      \includegraphics[scale=0.5]{"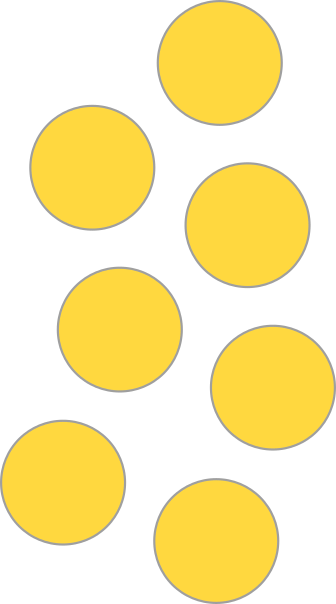"}} & $T^B_{2\leftrightarrow 3}$ & -10.798  \\

    \parbox[c]{1em}{
      \vspace*{0.1cm}
      \includegraphics[scale=0.5]{"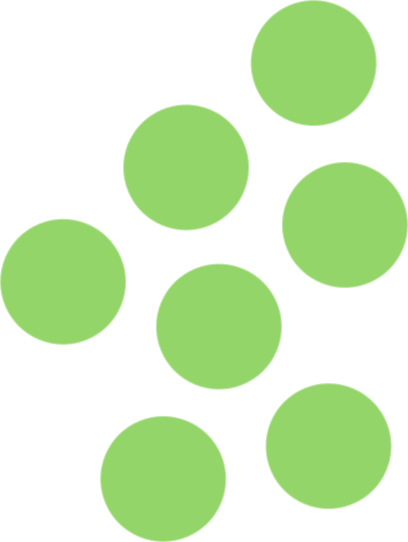"}} & $T^B_{1\leftrightarrow 2}$ & -10.898  &  

\parbox[c]{1em}
{
      \vspace*{0.1cm}
      \includegraphics[scale=0.5]{"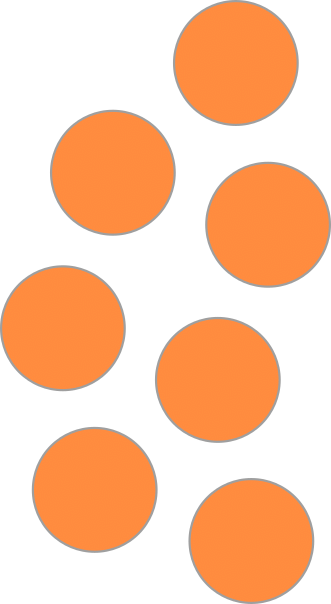"}} & $T^A_{2\leftrightarrow 3}$ & -10.771   \\
\parbox[c]{1em}
{
      \vspace*{0.1cm}
      \includegraphics[scale=0.5]{"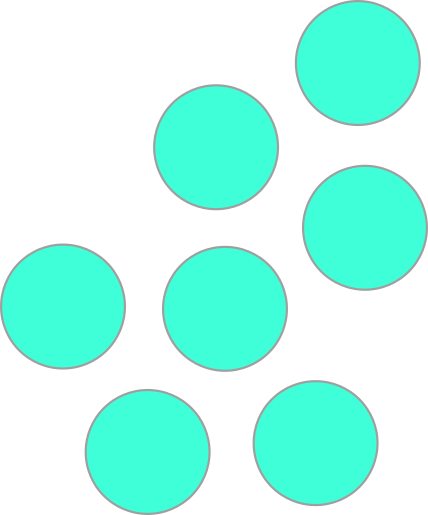"}} & $T^A_{2\leftrightarrow 2}$ & -10.882 & 
  \parbox[c]{1em}
{} &  &    \\
  \hline
  \end{tabular}
\end{minipage}

 \end{table}

The (known) local minima of the function~\eqref{eq:lj7} are shown in Table \ref{tab:lj7-min}.
Minima are denoted by $C_i$, $i=0\ldots,3$ (we used the same numbering as in the original work~\cite{dellago1998efficient}).
The transition states we identified are shown in Table \ref{tab:lj7-ts}. 
Transitions states are denoted by $T^{\bullet}_{i\leftrightarrow j}$, where $i\leftrightarrow j$ indicates that the transition state connects $C_i$ to $C_j$, and the superscript $\bullet$ is used to distinguish different transition states associated with the same local minima $C_i$ and $C_j$. 
The previous work~\cite{dellago1998efficient} identified 5 transition states ($T^A_{0\leftrightarrow2}, \ T^A_{1\leftrightarrow0}, \ T^B_{1\leftrightarrow1}, \ T^C_{1\leftrightarrow2}, \ T^D_{1\leftrightarrow2}$) \panos{and in \cite{wales2002discrete} a total of 19 index-1 saddle points were identified.
We were able to identify the 11 lowest index-1 saddle points. 
We note that the aim of this experiment is not to identify all the saddle points associated with the Lennard-Jones 7 atom potential, but to illustrate how the method can be applied to a non-trivial problem.
}
Let us emphasize that all the configurations (minima and saddle points) are identified up to translation, rotation, reflection symmetry and renumbering of the atoms: this is easily done since all these critical points have different energies, so that one energy can be associated to one critical point.

\begin{figure}
\centering
\includegraphics[width=0.8\textwidth]{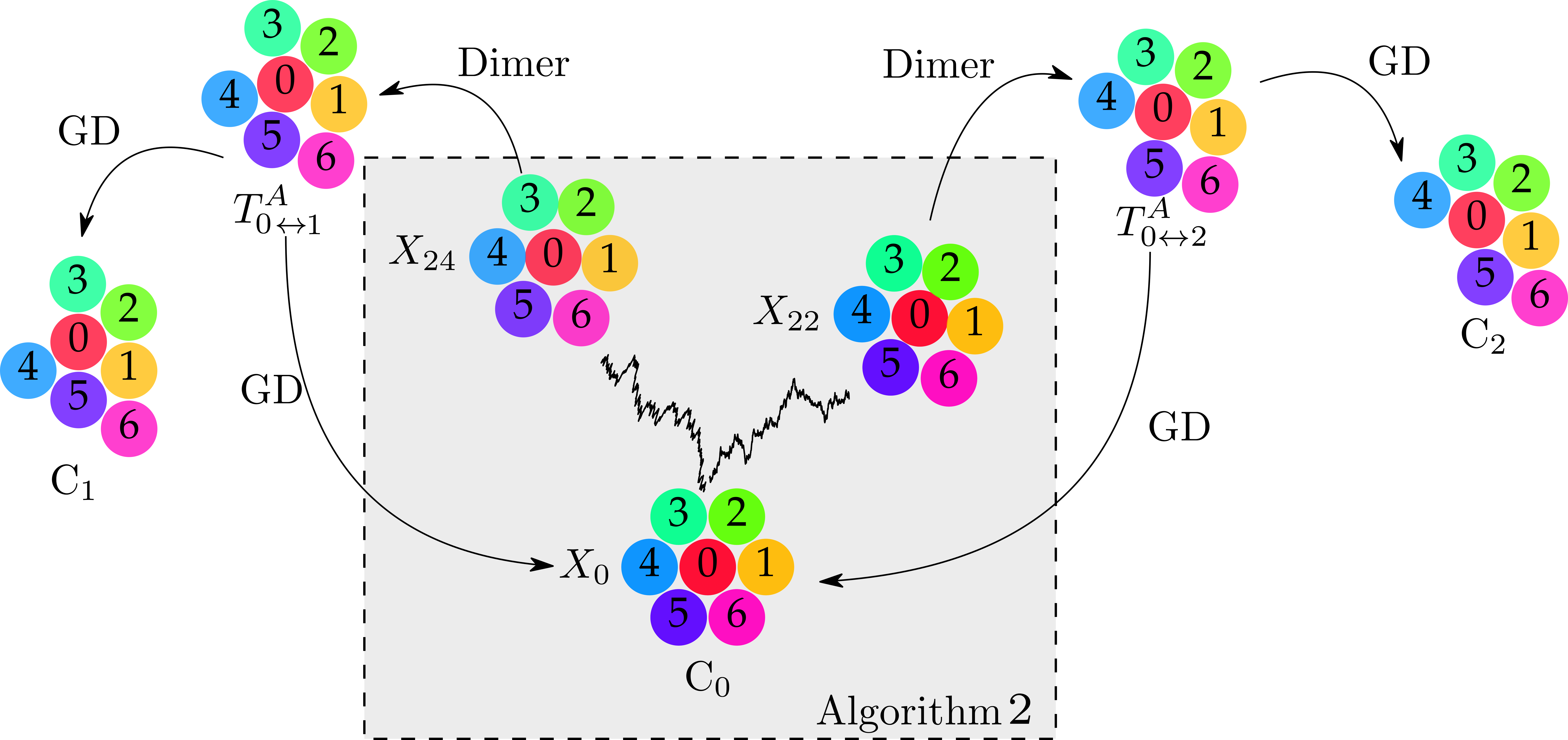}
\caption{\label{fig:c0_transitions}Three step process to identify the geometry of the Lennard Jones cluster (GD=gradient descent).}
\end{figure}

Figure \ref{fig:c0_transitions} illustrates the three step process we used to identify the local minima and transition states we report in Tables \ref{tab:lj7-min} and \ref{tab:lj7-ts}, using Algorithm~\ref{alg:capl} combined with Algorithm~\ref{alg:dimer} for the local search.
In the first step, we initialize Algorithm~\ref{alg:capl} from one of the local minima. In Figure~\ref{fig:c0_transitions}, we schematically represent a possible output of the algorithm when initialized at $C_0$ (which happens to be the global minimum of \eqref{eq:lj7}).
As specified in Algorithm \ref{alg:capl}, every $m$ iterations, we select the particle with the highest weight. For example, in Figure~\ref{fig:c0_transitions}, at iteration $k/m=22$ we selected the initial condition represented as $X_{22}$ on the figure. In the second phase, we perform a local dimer search using Algorithm \ref{alg:dimer}. In Figure \ref{fig:c0_transitions}, we see that the algorithm finds the transition state $T^A_{0\leftrightarrow 2}$.
In the third and final phase of our procedure, we run two gradient descent (GD) algorithms, initialized from the following two points,
\[
u_0^{GD}=u_{ts}\pm\gamma v_1,
\]
where $u_{ts}$ is the saddle point computed by Algorithm \ref{alg:dimer}, $\gamma$ is a small positive constant ($\gamma=0.01$ in all our experiments), and $v_1$ is the eigenvector associated with the smallest eigenvalue of the Hessian of the potential at the point $u_{ts}$.
The aim of this phase is to find the two local minima the transition state is connected to. We say that a saddle point connects two local minima when it lies at the intersection of the boundaries of the basins of attractions of these local minima for the steepest descent dynamics.  One could use the newly found local minimum as an initial condition for a new saddle point search (see e.g. \cite{mousseau2012activation}). 
As can be seen from Figure~\ref{fig:c0_transitions}, in this particular example one gradient descent trajectory converges back to $C_0$ (the initial condition of the saddle point search) and the other to $C_2$. 
In this case, we say that $C_0$ and $C_2$ are connected through $T^A_{0\leftrightarrow 2}$. 
Figure~\ref{fig:c0_transitions} also shows that at iteration $k/m=24$ a configuration is selected which leads to a different transition state ($T^A_{0\leftrightarrow 1}$) that connects $C_0$ to the local minimum $C_1$. Starting from $C_0$ we find two new local minima, $C_1$ and $C_2$. 
Notice that it may be possible to find transition states that are not associated with the initial condition $C_0$. In order to identify more saddle points, the whole procedure is then repeated starting from a new local minimum.

\begin{figure}
\centering
\includegraphics[width=0.4\textwidth]{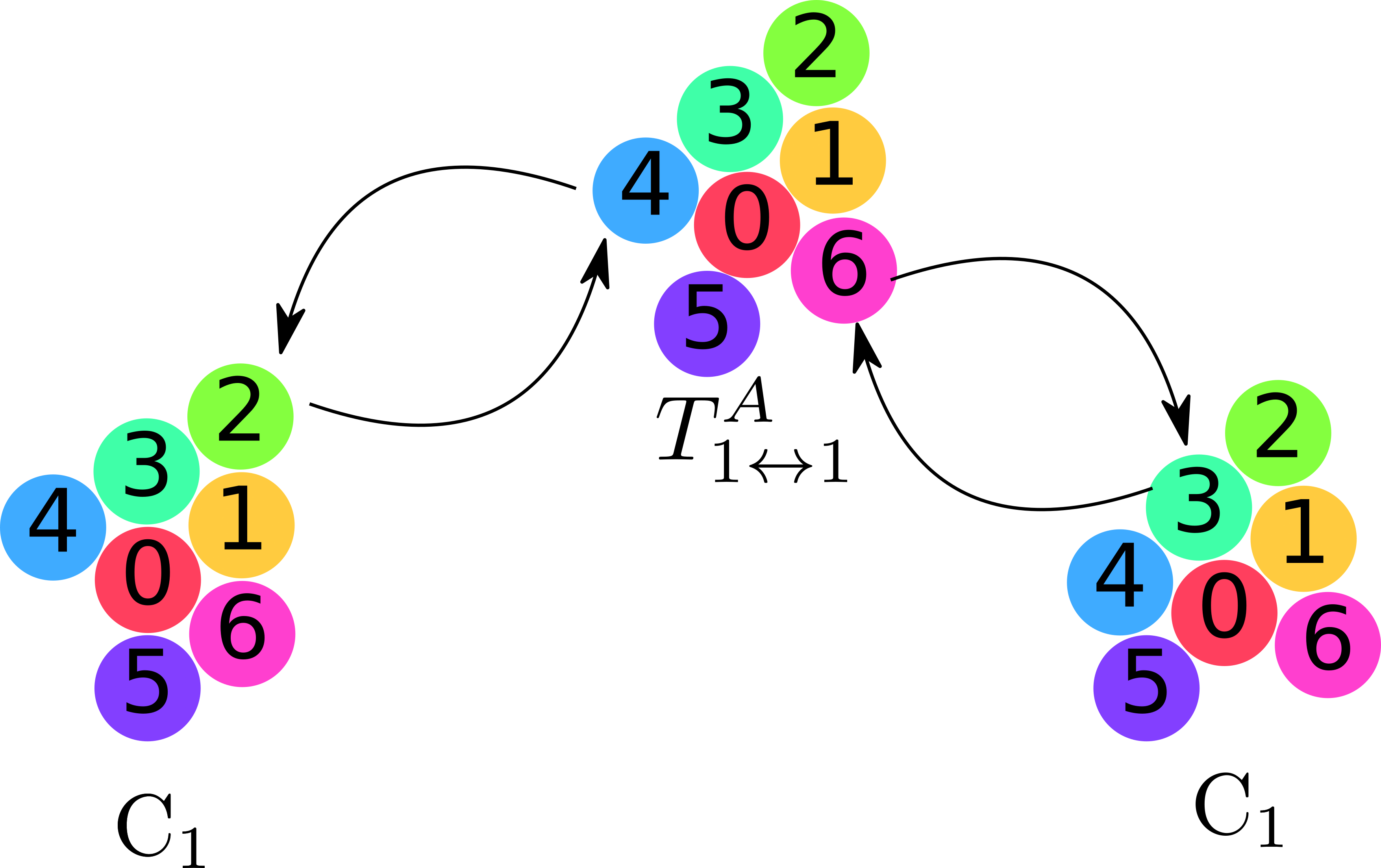}
\caption{\label{fig:c1_transitions} A transition from $C_1$ to $C_1$ through $T^A_{1 \leftrightarrow 1}$.}
\end{figure}

Let us emphasize that a transition state may connect two local minima which are actually the same up to a renumbering. Such a situation is shown in Figure 
\ref{fig:c1_transitions}: the transition state $T^{A}_{1\leftrightarrow 1} $ connects two $C_1$ configurations which are the same up to a reflection symmetry and a renumbering.

\begin{figure}
\begin{center}
\includegraphics[width=0.8\textwidth]{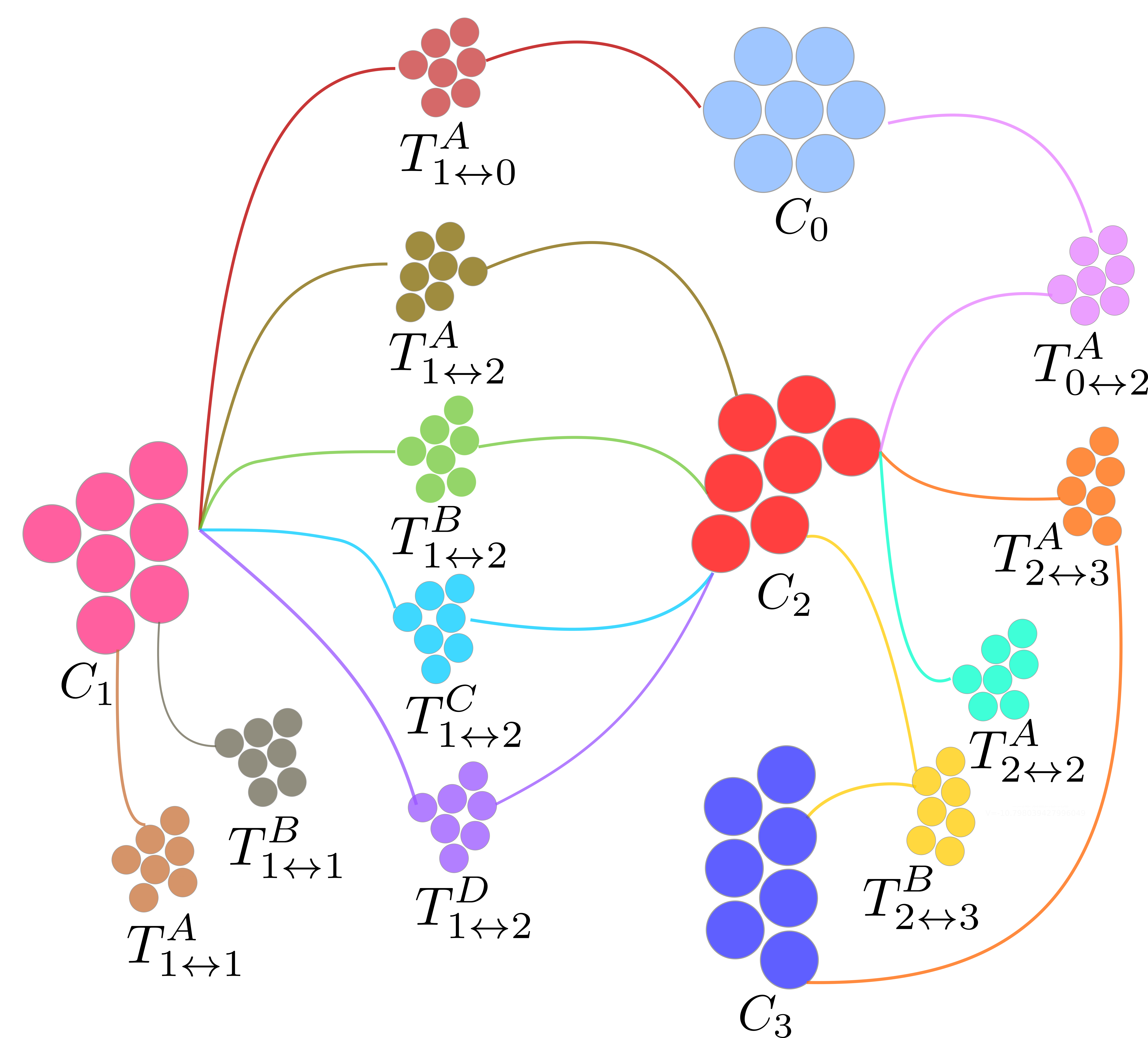}
\caption{\label{fig:lj7_graph} Undirected graph representing the possible transitions for the 7-atom Lennard-Jones cluster in 2d.}
\end{center}
\end{figure}

 We summarize the findings of our numerical experiments in the undirected graph of Figure \ref{fig:lj7_graph}, which represents all the local minima connected by all the transition states we identified.

The process described above can be performed with the classical dimer (Algorithm \ref{alg:dimer}) or the particle dimer (Algorithm \ref{alg:particle dimer}).
Just like we saw in the 2d numerical experiments with the M\"uller-Brown potential in Section \ref{sec: muller brown}, the advantage of the particle dimer method is that it can potentially identify more transition states than the simple dimer method.
 In the next experiment, we check whether this conclusion persists on the higher dimensional test case considered in this section. 
 \panos{
 We report in Figure \ref{fig:lj7 dimer success rate} the transition states identified with Algorithm \ref{alg:capl} initialized from the $C_1$ minimum.
 In these figures, a red dot denotes a transition state identified at iteration $k/m$. The transition states shown in red in the $y$-axis are the ones connected with $C_1$.
 In Figure \ref{fig:dimer iterations vanilla} 
 we use Algorithm~\ref{alg:capl} (SSPD-LS)  combined with Algorithm~\ref{alg:dimer} using the same procedure as described above, but with the resampling step disabled.  
 For Figure~\ref{fig:dimer iterations} and Figure \ref{fig:particle dimer iterations} we enable the resampling and run Algorithm \ref{alg:dimer} and Algorithm~\ref{alg:particle dimer} as the local search method respectively.
 From these results it is clear that the resampling step plays a crucial role in the proposed method. In addition, we observe that the particle dimer method is much more efficient. In fact, in this case, starting from $C_1$, Algorithm~\ref{alg:capl} combined with the particle dimer method (Algorithm~\ref{alg:particle dimer}) is able to identify all the transition states we found at once, without considering new runs from different local minima than $C_1$.}

\begin{figure}
\centering     

\subfigure[\panos{Local Search with Algorithm \ref{alg:dimer} \& without resampling} ]{\label{fig:dimer iterations vanilla}\includegraphics[width=0.4\textwidth]{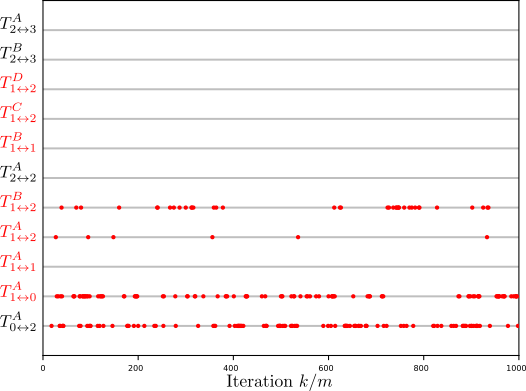}}\\
\subfigure[Local Search with Algorithm \ref{alg:dimer} ]{\label{fig:dimer iterations}\includegraphics[width=0.4\textwidth]{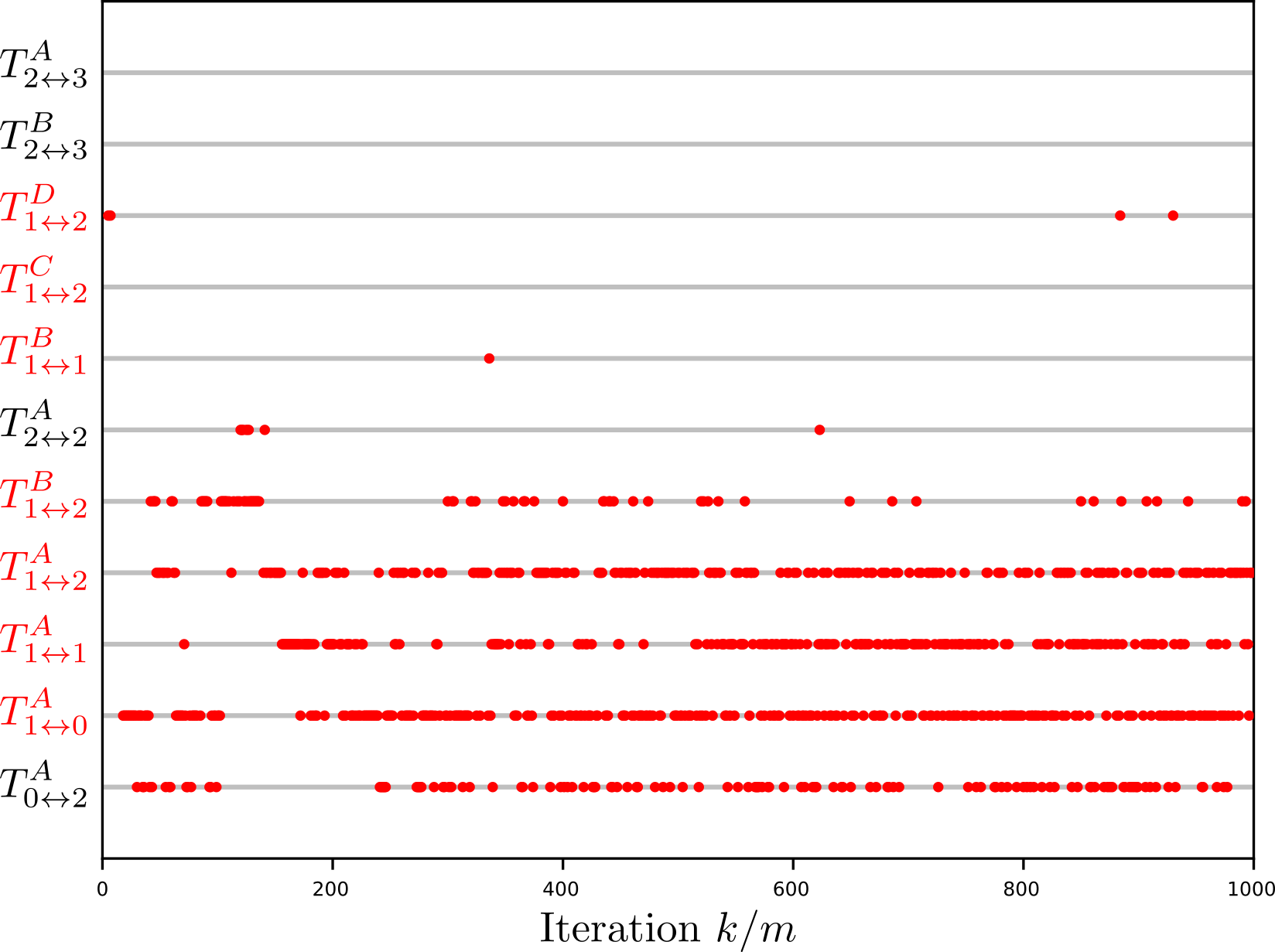}}
\subfigure[Local search with Algorithm \ref{alg:particle dimer}  ]{\label{fig:particle dimer iterations}\includegraphics[width=0.4\textwidth]{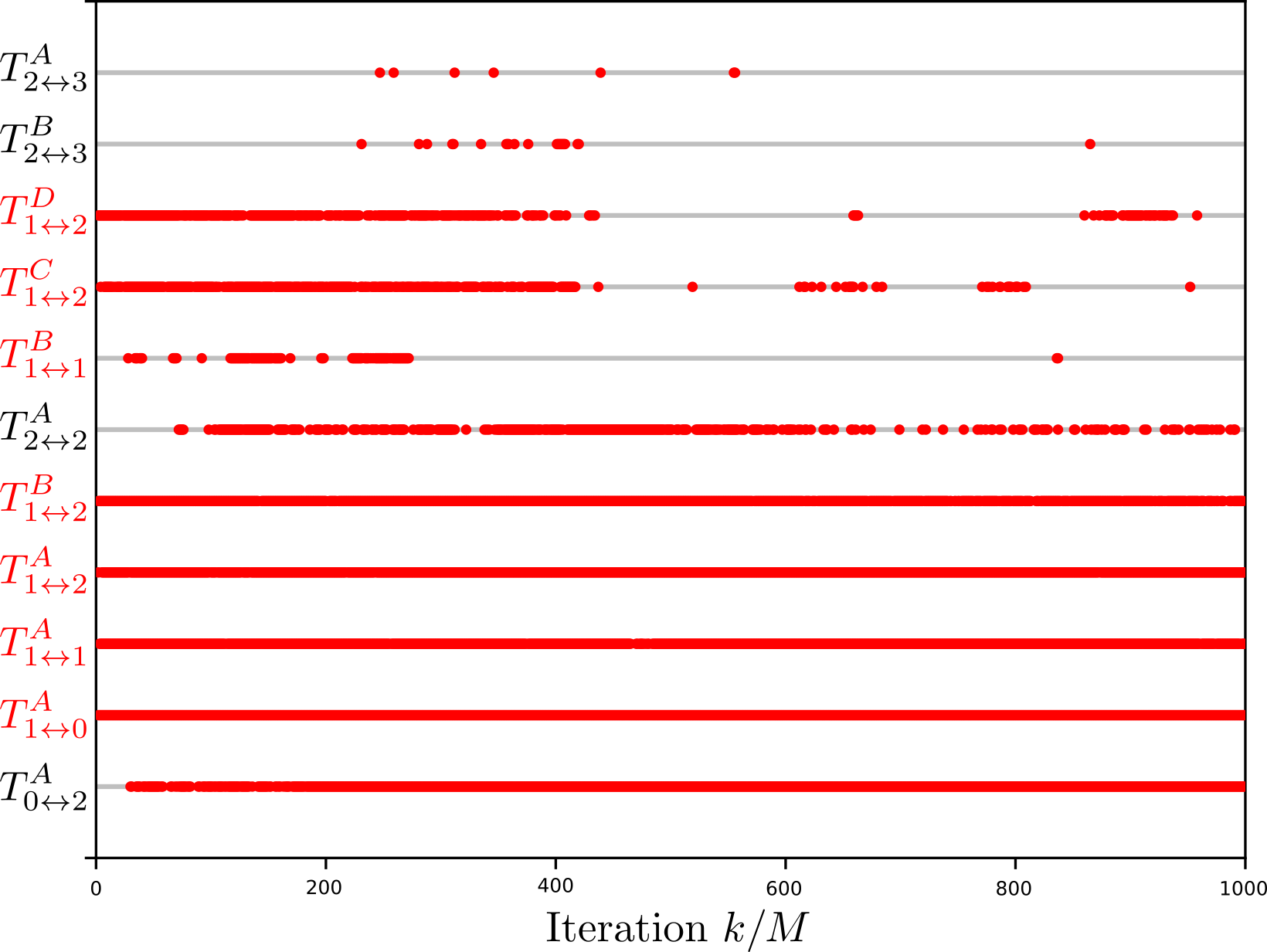}}
\caption{
Comparison of the performances of SSPD-LS combined with Algorithm \ref{alg:dimer} (Figure~\ref{fig:dimer iterations}) and Algorithm \ref{alg:particle dimer} (Figure~\ref{fig:particle dimer iterations}) for the local search of the transition states, on the 7-atom Lennard Jones cluster test case.
\panos{Figure \ref{fig:dimer iterations vanilla} shows that without the resampling step in SSDP-LS, the performance degrades a lot. }
}

\label{fig:lj7 dimer success rate}
\end{figure}

\section{Conclusions}\label{sec:conc}

In this article, we proposed a new algorithm to locate index-1 saddle points, based a probabilistic representation of the solution to the Witten partial differential equation~\eqref{eq:FP1} on $1$-forms. This algorithm can be seen as the natural extension to index-1 saddle points of the stochastic gradient descent method which locates local minima. We would like to explore in the future various extensions of the algorithm we propose, in particular to accelerate the discovery of many saddle points, in the spirit of enhanced sampling methods which are used in Markov Chain Monte Carlo algorithms to explore efficiently many local minima of a potential function (adaptive biasing methods, conditioning techniques, etc.)

\subsection*{Acknowledgements}
The authors would like to thank Antoine Levitt, Xue-Mei Li and Christoph Ortner for enlightening discussions on, respectively, saddle point search, stochastic differential geometry and benchmark test problems. The authors are also grateful to David Wales for very interesting feedbacks on an early version of this work, and to the anonymous reviewers for their useful suggestions and comments.
This work benefited from the support of the European Research Council under the European Union's Horizon 2020 research and innovation programme (grant agreement No 810367), project
EMC2. 
This work also benefited from the support of the project ANR QuAMProcs (ANR-19-CE40-0010) from the French National Research Agency. 
This project was initiated as TL was a visiting professor at Imperial College of London (ICL), with a visiting professorship grant from the Leverhulme Trust. 
The Department of Mathematics at ICL and the Leverhulme Trust are warmly thanked for their support.
The second author acknowledges support by JPMorgan Chase \& Co under J.P. Morgan A.I. Faculty Awards in 2019 and 2021. PP was also partially supported by the EPSRC through grant number EP/W003317/1.

\bibliographystyle{siamplain}
\bibliography{saddles_bib}

\clearpage
\section*{Supplementary Material}
~
\begin{table}[h]
\footnotesize
  \caption{Index of Animations}\label{tab:animation index}
\begin{center}
  \begin{tabular}{|l|l|} \hline
  Figure  & \bf File Name  \\ \hline
    Figure 1a & \url{https://www.doc.ic.ac.uk/~pp500/animations/2dpdeFP.mp4}  \\
    Figure 1b &  \url{https://www.doc.ic.ac.uk/~pp500/animations/2dpdeWitten.mp4} \\ 
    Figure 3 &  
    \url{https://www.doc.ic.ac.uk/~pp500/animations/2doublewell.mp4} \\ 
    \hline
  \end{tabular}
\end{center}
\end{table}

\end{document}